\documentclass[12pt,final]{article}

\usepackage{amsmath}
\usepackage{amsthm}
\usepackage{amssymb}
\usepackage{amsfonts}
\usepackage{latexsym}               
\usepackage{amsgen}
\usepackage[mathscr]{eucal}
\usepackage{mathrsfs}
\usepackage{enumerate}

\usepackage{mathptmx} 
\usepackage{fancybox}
\usepackage{color}
\usepackage{graphicx}

\usepackage{showkeys}
\usepackage{bm}

\newtheorem{theorem}{Theorem}
\newtheorem{proposition}[theorem]{Proposition}
\newtheorem{lemma}[theorem]{Lemma}

\theoremstyle{definition}
\newtheorem{definition}[theorem]{Definition}

\numberwithin{equation}{section}
\numberwithin{theorem}{section}

\def\pd{\partial}

\def\dt{\partial_t}

\newcommand{\tri}{{|\hspace{-0.5mm}|\hspace{-0.5mm}|}}

\usepackage{mathtools}

\begin{document}

\abovedisplayskip=8pt plus 2pt minus 4pt
\belowdisplayskip=\abovedisplayskip


\thispagestyle{plain}

\title{Boundary Layers of the Boltzmann Equation \\
in a Three-dimensional Half-space} 

\author{%
{\large\sc Shota Sakamoto${}^1$},
{\large\sc Masahiro Suzuki${}^2$},
\\
and
{\large\sc Katherine Zhiyuan Zhang${}^3$}
}

\date{ \today \\
\bigskip
\normalsize
${}^1$%
Department of Mathematics, 
Tokyo Institute of Technology,
\\
Meguro-ku, Tokyo, 152-8552, Japan
\\ [7pt]
${}^2$%
Department of Computer Science and Engineering, 
Nagoya Institute of Technology,
\\
Gokiso-cho, Showa-ku, Nagoya, 466-8555, Japan
\\ [7pt]
${}^3$%
Courant Institute of Mathematical Sciences, New York University, 
\\
New York, NY 10012, USA
} 

\maketitle

\begin{abstract}
We consider the nonlinear boundary layers of the
Boltzmann equation in a three-dimensional half-space by perturbing around a Maxwellian, under the assumption that the Mach number of the Maxwellian satisfies ${\cal M}_{\infty} < -1$. In \cite{UYY03} and \cite{UYY04},  nonlinear boundary layers of the Boltzmann equation in a half-line are considered, with stationary solutions obtained and nonlinear stability confirmed. In this paper, we establish the unique existence of stationary solutions for the three-dimensional half-space model, and show that the stationary solution is asymptotic stable.

\end{abstract}

\begin{description}
\item[{\it Keywords:}]
Boltzmann equation, boundary layers, time-periodic solution, stationary solution, 
asymptotic stability

\item[{\it 2020 Mathematics Subject Classification:}]
35B40; 
35Q20; 
76P05; 
82C40. 
\end{description}


\newpage

\section{Introduction}
The half-space problem of the kinetic equation has drawn a huge attention and been extensively studied for decades due to its connection with physical phenomena in a very small mean free path regime.
In particular, research on the nonlinear Boltzmann equation in the half-space with various boundary conditions and collision kernels has been conducted for nearly a score of years. As we will see later, most of preceding results focused on the case when the space variable is in $\mathbb{R}_+$, assuming slab-symmetry. Our aims are to analyze the Dirichlet problem without this symmetry and to find a solution in the three-dimensional half-space $\mathbb{R}^3_+$ with its asymptotic profile.

We consider an asymptotic behavior of a solution to the Boltzmann equation
\begin{subequations}
\label{boltz0}
\begin{gather}
\partial_{t} F+\xi\cdot \nabla_x F=Q(F,F), \quad t>0, \  x \in {\mathbb R}^{3}_{+} , \ \xi \in \mathbb R^{3},
\label{boltz1}
\end{gather}
where ${\mathbb R}^{3}_{+} := \{ x{=(x_1, x_2, x_3)} \in \mathbb{R}^3\ \textcolor{blue}{|}\ x_1 > 0 \} $.
The collision operator $Q$ is a bilinear integral operator
\begin{align*}
Q(F,G):=\frac{1}{2} \int_{{\mathbb R^{3}}\times \bm{S}^2} 
&(F(\xi')G(\xi_{*}')+F(\xi_{*}')G(\xi')-F(\xi)G(\xi_{*})-F(\xi_{*})G(\xi))
\notag \\
&\quad \times q(\xi-\xi_{*},\omega)d\xi_{*}d\omega
\end{align*}
with 
\begin{equation*}
\xi'=\xi-[(\xi-\xi_{*})\cdot \omega] \omega,
\quad 
\xi'_{*}=\xi_{*}+[(\xi-\xi_{*})\cdot \omega] \omega.
\end{equation*}
We restrict ourselves to the hard sphere gas for which the collision {kernel} $q=q(\zeta,\omega)$ is given by
\begin{equation*}
q(\zeta,\omega)=\sigma_{0}|\zeta\cdot\omega|,
\end{equation*}
where {$\sigma_0$} is the surface area of the hard sphere.

In this equation, $t>0$,  $x = (x_1, x_2, x_3) = (x_1, x') \in {\mathbb R}^{3}_{+} $, and $\xi = (\xi_{1},\xi_{2},\xi_{3}) \in \mathbb R^{3}$ are the time variable, space variable, and {velocity variable}, respectively.
{The} unknown function $F = F(t,x,\xi)$ stands for the mass density of gas particles.
We put an initial condition
\begin{gather}
F (0,x,\xi) = F_{0}(x,\xi), \quad x \in {\mathbb R}^{3}_{+} , \ \xi \in \mathbb R^{3}
\label{ice}
\end{gather}
and boundary conditions
\begin{gather}
F (t,0,x',\xi) = F_{b}(t,x',\xi), \quad
t>0, \ x' \in \mathbb R^{2}, \ \xi_{1}>0,
\label{bce1}
\\
\lim_{x_{1}\to\infty} F (t,x_{1},x',\xi) \to  M_{\infty}(\xi) \ \ (x_{1} \to \infty), 
\quad t>0, \ x' \in \mathbb R^{2}, \ \xi \in \mathbb R^{3}.
\label{bce2}
\end{gather}
\end{subequations}
Here $M_{\infty}$ is a Maxwellian with constants $\rho_{\infty}>0$, 
$u_{\infty}=(u_{\infty,1},u_{\infty,2},u_{\infty,3}) \in \mathbb R^{3}$, and $T_{\infty}>0$:
\begin{equation*}
M_{\infty}(\xi)=M[\rho_{\infty},u_{\infty},T_{\infty}](\xi)
:=\frac{\rho_{\infty}}{(2\pi T_{\infty})^{3/2}}\exp\left(-\frac{|\xi-u_{\infty}|^{2}}{2T_{\infty}} \right).
\end{equation*}
It is well-known the fact $Q(M_\infty, M_\infty)=0$, hence $M_{\infty}$ satisfies equation \eqref{boltz1}.

Throughout this paper,  
we assume that the equilibrium state $(\rho_{\infty},u_{\infty},T_{\infty})$ satisfies
\begin{equation*}\label{super1}
u_{\infty,2}=u_{\infty,3}=0, \quad 
{\cal M}_{\infty}:=\frac{u_{\infty,1}}{\sqrt{5T_{\infty}{/3}}}<-1,
\end{equation*}
where ${\cal M}_{\infty}$ {stands for} the Mach number of the equilibrium state.
Furthermore, we use the weight function
\begin{equation*}
W_{0}(\xi) := \left( M[1,u_{\infty},T_{\infty}](\xi) \right)^{1/2}
\end{equation*}
and seek the {{solutions}} of \eqref{boltz0} in the form
\begin{equation*}
F(t,x,\xi)=M_{\infty}(\xi) + W_{0}(\xi)f(t,x,\xi).
\end{equation*}
The initial-boundary value problem \eqref{boltz0} can be rewritten as
\begin{subequations}\label{ree0}
\begin{gather}
\partial_{t} f + \xi\cdot \nabla_x f - Lf = \Gamma (f), \quad t>0, \  x \in {\mathbb R}^{3}_{+} , \   \xi \in \mathbb R^{3},
\label{ree1}\\
f(0,x,\xi) = f_{0}(x,\xi) :=  W_{0}^{-1} (F_{0}(x,\xi) - M_{\infty}(\xi)), \quad x \in {\mathbb R}^{3}_{+} , \   \xi \in \mathbb R^{3},
\label{reic1}\\
f(t,0,x',\xi) =  f_{b}(t,x',\xi) := W_{0}^{-1} (F_{b}(t,x',\xi) - M_{\infty}(\xi)), \quad t>0, \  x' \in \mathbb R^{2}, \ \xi_{1} >0,
\label{rebc1}\\
\lim_{x_{1}\to\infty} f (t,x_{1},x',\xi) \to  0 \ \ (x_{1} \to \infty),  \quad t>0, \ x' \in \mathbb R^{2}, \ \xi \in \mathbb R^{3},
\label{rebc2}
\end{gather}
\end{subequations}
where
\begin{gather*}
Lf:=W_{0}^{-1}\left\{ Q(M_{\infty},W_{0}f)+Q(W_{0}f,M_{\infty}) \right\},  \quad
\Gamma(f):=\Gamma(f,f), \quad \Gamma(f,g):=W_{0}^{-1} Q(W_{0}f,W_{0}g).
\end{gather*}
In this paper, we study the unique existence and asymptotic stability of
stationary and time-periodic solutions of \eqref{ree0} for the cases 
that the boundary data $f_{b}$ are time-independent and time-periodic, respectively.
Here and hereafter we discuss the solutions in the sense of distributions unless otherwise noted.

As the problem is settled, let us review the known literature on the half-space problem of the Boltzmann equation. For more detailed reference (including numerical studies which brought the importance of the pair of boundary data~\eqref{bce1}-\eqref{bce2} to light), the reader may consult~\cite{ANSS01, SAY01, BGS01, Sone01, Sone02} and reference therein.

We first review {{the ones analyzing an}} one-dimensional linearized steady problem
\begin{align}\label{LinearProblem}
\xi_1 \partial_{x_1} f-Lf=0,\quad f=f(x_1, \xi).
\end{align}
For $\mathscr{M}_\infty=0$,
Bardos, Caflisch, and Nicolaenko~\cite{BCN01} established the existence of a unique solution to the Milne and Kramers problems for the hard sphere model. These are Dirichlet problems of \eqref{LinearProblem} with specified bulk velocity (or specified asympotics of macroscopic quantities). They also showed exponential decay of the solution to some Maxwellian at far-field ($x_1\rightarrow \infty$).
This result was generalized by Golse and Poupaud~\cite{GP01} for the Maxwellian molecule and moderately soft potential
cases, that is, those when
\begin{align}\label{GeneralPotential}
q(\zeta, \omega)=b\bigg(\frac{\zeta}{\vert \zeta\vert}\cdot \omega \bigg)\vert \zeta\vert^\gamma,\quad   \gamma\in [-2, 0],\quad 0\le b\bigg(\frac{\zeta}{\vert \zeta\vert}\cdot \omega\bigg)\in L^1(\mathbf{S}^2).
\end{align}
Also, Coron, Golse, and Sulem~\cite{CGS01} showed that, for any $\mathscr{M}_\infty$,
there exists a unique solution to \eqref{LinearProblem} for any incoming boundary data satisfying some compatibility conditions and it converges to some Maxwellian at infinity. Note that the number of the conditions are completely characterized by the value of $\mathscr{M}_\infty$, which was conjectured by Cercignani~\cite{Cercignani01}.

After these early results on the linear problem,
Golse, Perthame, and Sulem~\cite{GPS01} {{analyzed}} a nonlinear problem
\begin{align*}
\xi_1 \partial_{x_1} f-Lf=\Gamma(f),\quad f=f(x_1, \xi)
\end{align*}
with the specular boundary condition for the hard sphere model.  However, the method developed in this paper is not applicable to the case of the Dirichlet boundary condition. 
Arkeryd and Nouri~\cite{AN01} studied the fully nonlinear (i.e.\ non-perturbed) Dirichlet problem with a truncated kernel to further study the Milne problem. 
We also remark that there are similar existence and stability results for the discrete velocity model~\cite{Ukai01, KN01, KN02, KNik01}.

After these results, Ukai, Yang, and Yu~\cite{UYY03} succeeded in proving the existence of a unique steady solution to \eqref{ste0} below, where a far-field Maxwellian is initially given, and the set of compatible incoming boundary data forms a $C^1$ manifold whose co-dimension is determined by $\mathscr{M}_\infty\neq 0, \pm 1$ (i.e. non-degenerate values). This solution exponentially converges to the far-field Maxwellian at infinity, which was shown by using the weight function $e^{\sigma x_1}$ to exploit an artificial damping term.  They also showed~\cite{UYY04} that a solution to the corresponding time-dependent problem exponentially converges to the stationary solution if $\mathscr{M}_\infty<-1$, when no compatibility condition on boundary data is required. Later existence for the degenerate cases $\mathscr{M}_\infty=0, \pm 1$ were proved by Golse~\cite{G01}.

Followers of these results then generalized and improved them in various ways: Chen, Liu, and Yang~\cite{CLY01} proved the unique existence of the problem for the cut-off hard potential case, that is, the kernel $q$ is given as in \eqref{GeneralPotential} with $\gamma \in [0,1]$
 (we remark that \cite{CLY01, WYY01} consider slightly more general cases). The stability of a solution with algebraic decay for this case was shown by Wang, Yang, and Yang~\cite{WYY01}. A key idea for these results was to modify the weight function $e^{\sigma x}$ to $e^{\sigma(x,\xi) x}$, where $\sigma(x, \xi)$ is chosen so that it copes with the growth of $\vert \cdot \vert^\gamma$. A similar idea was applied to solve the case of moderately soft potential~\cite{WYY02, Yang01},
 and finally the full range $\gamma\in (-3,1]$ was treated in a unified way for an existence result~\cite{Yang02}. In addition to the study of the problem with the Dirichlet boundary condition, the analysis of ones with various boundary conditions   have been also conducted~\cite{Tian01, TS01, ST01, ST02}.
 
It should be also mentioned that, instead of using the energy method developed in the papers above, Liu and Yu~\cite{LY03} gave the very detailed description of invariant manifolds by finely {{analyzing}} the Green function of the linearized problem (see also~\cite{LY01, LY02}), and as an application they studied the boundary layer.

All of the research above (except~\cite{LY02}) assumed slab-symmetry, meaning
the solutions are independent of the tangential direction $x'$.
This assumption is physically reasonable.  
However, it is of greater interest to consider a situation that
the solutions depend on all directions $(x_{1},x')$.
Our main results below do not require the slab-symmetry.


Among the aforementioned results, we mention the result of \cite{UYY03} {as a proposition,}
since it plays an essential role in our analysis. 
{{In \cite{UYY03}, the unique existence of stationary solutions $\tilde{f}=\tilde{f}(x_{1},\xi)$
over the half-space $\mathbb R_+:=\{ x_{1}>0\}$ was shown.}}
The stationary solution solves a boundary value problem
\begin{subequations}\label{ste0}
\begin{gather}
\xi_{1} \partial_{x_{1}}\tilde{f} - L\tilde{f} = \Gamma (\tilde{f}), \quad x_{1}>0, \   \xi \in \mathbb R^{3},
\label{ste1}\\
\tilde{f}(0,\xi) =  a_{0}(\xi),  \quad \xi_{1}>0,
\label{stbc1}\\
\lim_{x_{1}\to\infty} \tilde{f} (x_{1},\xi) \to  0 \ \ (x_{1}\to \infty),  \quad   \xi \in \mathbb R^{3}.
\label{stbc2}
\end{gather}
\end{subequations}
\begin{proposition}[\cite{UYY03}] \label{ex-st}
Let $\beta> 3/2$. There exist positive constants $\delta_{0}$, $\gamma_{0}$, and $M_{0}$ such that if $\tilde{\delta} \leq \delta_{0}$ and $|a_{0}(\xi)| \leq \tilde{\delta} (1+|\xi|)^{-\beta}$, then the problem \eqref{ste0} has a unique solution $\tilde{f}$ satisfying
\begin{equation*}
|\tilde{f}(x_{1},\xi)| \leq \tilde{\delta} M_{0} e^{-\gamma_{0} x_{1}}(1+|\xi|)^{-\beta}.
\end{equation*}
\end{proposition}

\subsection{Notation} \label{ss-Notation}

We introduce notation used often in this paper.
For {a domain} $\Omega \subset \mathbb R^{n}$ and $1 \leq p \leq \infty$,
the space $L^p(\Omega)$ denotes the standard $L^{p}$-Lebesgue spaces.
We abbreviate $L^p(\mathbb R^{3}_{+} \times \mathbb R^{3})$, $L^p(\mathbb R^{3})$,
and $L^p(\mathbb R \times \mathbb R^{3}_{+} \times \mathbb R^{3})$
{by} $L^{p}$, $L^{p}_{\xi}$, and $L^{p}_{t,x,\xi}$, respectively. 
Note that  $\|\cdot\|:=\| \cdot \|_{L^{2}}$.
In the case $p=2$, { the spaces $L^{2}$, $L^{2}_{\xi}$, and $L^{2}_{t,x,\xi}$ are Hilbert spaces 
with the inner products $(\cdot, \cdot)_{L^2}$, $(\cdot, \cdot)_{L^2_{\xi}}$, and $(\cdot,\cdot)_{L^{2}_{t,x,\xi}}$, }respectively.
For 
$\beta > 0$, 
$L^{\infty}_{\beta}(\mathbb R^{3})$ denotes for 
the weighted $L^{\infty}$-Lebesgue space defined as
\begin{gather*}
L^{\infty}_{\beta}(\mathbb R^{3}):=\left\{f \in L^{\infty}(\mathbb R^{3}) \ \left| \ \|f\|_{L^{\infty}_{\beta}(\mathbb R^{3})}:=\sup_{\xi \in \mathbb R^{3}} |\langle \xi \rangle^{\beta}f(\xi)| <+\infty \right\}\right.,
\end{gather*}
where $\langle \xi \rangle:=1+|\xi|$. 
For $\Omega\subset \mathbb{R}^3$, we similarly define $L^\infty_\beta(\Omega\times \mathbb{R}^3)$,  the velocity-weighted Lebesgue space over the phase space.
We abbreviate $L^{\infty}_{\beta}(\mathbb R^{3}_{+} \times \mathbb R^{3})$ and $L^{\infty}_{\beta}(\mathbb R^{3})$
{by} $L^{\infty}_{\beta}$ and $L^{\infty}_{\xi,\beta}$, respectively, and also write those norm as $\|\cdot\|_{\beta}$ and $\|\cdot\|_{\xi,\beta}$.
Furthermore, we use the function space $L^{\infty}_{x}L^{2}_{\xi}$:
\begin{gather*}
L^{\infty}_{x}L^{2}_{\xi}:=\left\{f \in L^{1}_{\rm loc}(\mathbb R^{3}_{+} \times \mathbb R^{3}) \ | \ \|f\|_{L^{\infty}_{x}L^{2}_{\xi}} <+\infty \right\},
\quad 
\|f\|_{L^{\infty}_{x}L^{2}_{\xi}}:=\sup_{x \in \Omega} \| f (x,\cdot)\|_{L^{2}_{\xi}}.
\end{gather*}
Note that $L^{\infty}_{x}L^{2}_{\xi} \supset L^{\infty}_{\beta}$ holds if $\beta>3/2$.
For $t_{0} \in \mathbb R$, $\kappa \geq 0$, and $\beta>0$, 
we define the solution space $X(t_{0},\kappa,\beta)$ as
\begin{align*}
X({t_{0},\kappa,\beta})&:=\{f \in  L^{\infty}(t_{0},\infty;L^{\infty}_{\beta}) \cap C([t_{0},\infty) ; L^{2})\, | \, \tri f\tri_{t_{0},\kappa,\beta} <\infty\}, 
\\
\tri f \tri_{t_0,\kappa,\beta}&:=\sup_{t_0 \leq \tau <\infty} e^{\kappa \tau } [[f(\tau)]]_{\beta}, \quad
\text{where $[[\ \cdot \ ]]_{\beta}:=\|\cdot\| +\|\cdot\|_{^{\beta}}$}.
\end{align*}

Let $\pd_{x_{i}} := \frac{\pd}{\pd x_i}$ and $\dt := \frac{\pd}{\pd t}$.
The operator $\nabla_{x} := (\pd_{x_{1}} ,\pd_{x_{2}},\pd_{x_{3}})$
denotes {{the standard gradient with respect to $x = (x_1,x_2,x_3)$}}.
The operator $\nabla_{x'} := (\pd_{x_{2}},\pd_{x_{3}})$ denotes {{the tangential gradient with respect to $x'=(x_2,x_3)$}}.
We use $c$ and $C$ to denote generic positive constants 
depending only on $\sigma_{0}$, $\rho_{\infty}$, $u_{\infty}$, $T_{\infty}$, $\beta$, $\gamma_{0}$, $M_{0}$, and $R$, 
where $R$ is a positive constant being in \eqref{bc1}.
Let us also denote a generic positive constant {additionally depending} 
on other parameters $a$, $b$, $\ldots$ by $C(a,\, b,\, \ldots)$.
Furthermore, $A \lesssim B$ means $A \leq C B$ for the generic constant $C$ given above.

\subsection{Main results}
This subsection provides our main results on the stationary and time-periodic solutions.
We first discuss the stationary solutions $f^{s}(x,\xi)$ over the half-space ${\mathbb R}^{3}_{+} $
by regarding $f^{s}(x,\xi)$ as a perturbation of $\tilde{f}(x_{1}, \xi)$.
The stationary solutions satisfy 
\begin{subequations}\label{sree0}
\begin{gather}
 \xi\cdot \nabla_x f^{s} - Lf^{s} = \Gamma (f^{s}), \quad   x \in {\mathbb R}^{3}_{+} , \   \xi \in \mathbb R^{3},
\label{sree1}\\
f^{s}(0,x',\xi) =  f_{b}(x',\xi), \quad  x' \in \mathbb R^{2}, \ \xi_{1} >0,
\label{srebc1}\\
\lim_{x_{1}\to\infty} f^{s} (x_{1},x',\xi) \to  0 \ \ (x_{1} \to \infty),  \quad  x' \in \mathbb R^{2}, \ \xi \in \mathbb R^{3}.
\label{srebc2}
\end{gather}
\end{subequations}
To make assumptions on the boundary data $f_{b}$, let us set 
\begin{equation}\label{bc1}
f_{b}^{s}(x',\xi):=f_{b}(x',\xi)-\varphi_{R}(x')a_{0}(\xi)
\end{equation}
for $a_{0}$ being in Proposition \ref{ex-st} 
and $\varphi_{R}(x'):=\varphi (|x'|/R)$ with $R>0$, where $\varphi \in C^{\infty}_{0}(\overline{\mathbb R_{+}})$
is a non-negative function with  $\varphi(s)=1$ if $s \leq 1$, $\varphi(s)=0$ if $s \geq 2$.
Now we assume that the boundary condition ${f}_{b}$ is time-independent and satisfies
\begin{gather}\label{bc_assump1}
\langle\xi\rangle {f}_{b}^{s} \in L^{2}(\mathbb R^{2}\times\mathbb R^{3}), \quad
{f}_{b}^{s} \in L^{\infty}_{\beta}(\mathbb R^{2}\times\mathbb R^{3}).
\end{gather}
For the notational convenience, we use
\begin{equation*}
\delta^{s}:=\tilde{\delta}+\|\langle\xi\rangle {f}_{b}^{s}  \|_{L^{2}(\mathbb R^{2}\times\mathbb R^{3})} + \| {f}_{b}^{s} \|_{L^{\infty}_{\beta}(\mathbb R^{2}\times\mathbb R^{3})},
\end{equation*}
where $\tilde{\delta}$ is being in Proposition \ref{ex-st}.

To state {{the}} main theorems, we introduce an extension of the boundary data: 
\begin{equation*}
F_{b}(x,\xi):=\chi(\xi_{1}){f}_{b}^{s}(x_{2}-x_{1}\xi_{2}/\xi_{1},x_{3}-x_{1}\xi_{3}/\xi_{1},\xi),
\end{equation*}
where $\chi(s)$ is the 1D characteristic function of the set $\{ s > 0 \}$.
Note that $F_{b}(x,\xi)$ satisfies
\begin{gather*}
\xi \cdot \nabla_{x} F_{b}=0, \quad x \in \mathbb R^{3}_{+}, \ \xi \in \mathbb R^{3},
\\
F_{b}(0,x',\xi) =  {f}_{b}^{s}(x',\xi), \quad  x' \in \mathbb R^{2}, \ \xi_{1} >0.
\end{gather*}
We also use the function $U^{s}(x,\xi):=\varphi(x_{1}) F_{b}(x,\xi)$ which satisfies
\begin{subequations}\label{ExBdry0}
\begin{gather}
\xi \cdot \nabla U^{s}= F_{b} \xi_{1} \partial_{x_{1}} \varphi,
\quad x \in \mathbb R^{3}_{+}, \ \xi \in \mathbb R^{3},
\label{ExBdry4} \\
U^{s}(0,x',\xi)={f}_{b}^{s}(x',\xi),  \quad  x' \in \mathbb R^{2}, \ \xi_{1} >0,
\label{ExBdry1} \\
U^{s}(x_1,x',\xi)=0, \quad x_1 >2, \ x' \in \mathbb R^{2}, \ \xi \in \mathbb R^{3}.
\label{ExBdry2} 
\end{gather}
It is straightforward to check that
\begin{equation}
\|(\langle\xi\rangle U^{s}, \langle\xi\rangle (\partial_{x_{1}} \varphi_{1}) F_{b})\|  + \|(U^{s},F_{b})\|_{\beta}
\lesssim \delta^{s}.
\label{ExBdry3}
\end{equation}
\end{subequations}
The existence and stability results on time-periodic solutions 
are summarized in the following theorems. We use the solution space $X(0,0,\beta)$, where
\begin{gather*}
X({t_{0},\kappa,\beta})=\{f \in  L^{\infty}(t_{0},\infty;L^{\infty}_{\beta}) \cap C([t_{0},\infty) ; L^{2} )\, | \, \tri f\tri_{t_{0},\kappa,\beta} <\infty\} \quad \text{for $t_{0}\in\mathbb R$, $\kappa \geq 0$},
\\
\tri f \tri_{t_0,\kappa,\beta}=\sup_{t_0 \leq \tau <\infty} e^{\kappa \tau } [[f(\tau)]]_{\beta}, \quad
[[\ \cdot \ ]]_{\beta}=\|\cdot\| +\|\cdot\|_{^{\beta}}.
\end{gather*}

\begin{theorem}\label{th1}
Let $\beta>7/2$ and $R>0$.  
Assume that the boundary data $f_{b}$ is time-independent and satisfies \eqref{bc_assump1}.
There {exist} constants $\sigma>0$, $\eta_{0}>0$,  and $C_0>0$
such that if $\delta^{s} \leq \eta_{0}$, the stationary problem \eqref{sree0} has
a unique solution $f^{s}$ as $e^{\sigma x_{1}}(f^{s}-\varphi_{R}\tilde{f}-U^{s}) \in X(-\infty,0,\beta)$ and
$[[ e^{\sigma x_{1}}(f^{s}-\varphi_{R}\tilde{f}-U^{s}) ]]_{\beta} \leq C_{0}\delta^{s}$.
\end{theorem}

\begin{theorem}\label{th2}
Let $\beta>7/2$ and $R>0$. 
Assume that the boundary data $f_{b}$ is time-independent and satisfies \eqref{bc_assump1}.
There {exist} constants $\sigma>0$, $\eta_{0}>0$,  and $C_0>0$
such that if $[[e^{\sigma x_{1}}(f_{0}-\varphi_{R}\tilde{f}-U^{s})]]_{\beta}+\delta^{s} \leq \eta_{0}$, 
the initial-boundary value problem \eqref{ree0} has
a unique solution $f$ as $e^{\sigma x_{1}}(f-\varphi_{R}\tilde{f}-U^{s}) \in X(0,0,\beta)$ and
$\tri e^{\sigma x_{1}}(f-\varphi_{R}\tilde{f}-U^{s}) \tri_{0,0,\beta} 
\leq C_{0}([[e^{\sigma x_{1}}(f_{0}-\varphi_{R}\tilde{f}-U^{s})]]_{\beta}+\delta^{s})$.
Moreover, it holds that
\begin{equation*}
[[e^{\sigma x_{1}}(f-f^{s})(t)]]_{\beta} \leq C e^{-\kappa t/2} \quad \text{for $t>0$},
\end{equation*}
where $\kappa$ and $C$ are some positive constants.
\end{theorem}

Next we discuss the time-periodic solutions $f^{*}$ with a period $T^{*}$ to
\begin{subequations}\label{tree0}
\begin{gather}
\partial_{t} f^{*} + \xi\cdot \nabla_x f^{*} - Lf^{*} = \Gamma (f^{*}), \quad  t \in \mathbb R, \ x \in {\mathbb R}^{3}_{+} , \   \xi \in \mathbb R^{3},
\label{tree1}\\
{f}^{*}(t+T^{*},x,\xi)={f}^{*}(t,x,\xi), \quad  t \in \mathbb R, \ x \in {\mathbb R}^{3}_{+} , \   \xi \in \mathbb R^{3},
\\
f^{*}(t,0,x',\xi) =  f_{b}(t,x',\xi), \quad  t \in \mathbb R,  \ x' \in \mathbb R^{2}, \ \xi_{1} >0,
\label{trebc1}\\
\lim_{x_{1}\to\infty} f^{*} (t,x_{1},x',\xi) \to  0 \ \ (x_{1} \to \infty),  \quad  x' \in \mathbb R^{2}, \ \xi \in \mathbb R^{3},
\label{trebc2}
\end{gather}
\end{subequations}
assuming that
the boundary data $f_{b}$ have higher regularity than \eqref{bc_assump1}. 
Similar to \eqref{bc1}, let us set
\begin{equation*}
{f}_{b}^{*}(t,x',\xi):=f_{b}(t,x',\xi)-\varphi_{R}(x')a_{0}(\xi)
\end{equation*}
and then assume that the boundary condition ${f}_{b}$ satisfies
\begin{subequations}\label{bc_assump2}
\begin{gather}
{f}_{b}^{*}(t+T^{*},x,\xi)={f}_{b}^{*}(t,x,\xi), \quad  t \in \mathbb R, \ x \in {\mathbb R}^{3}_{+} , \   \xi \in \mathbb R^{3},
\\
\langle\xi\rangle {f}_{b}^{*}, \partial_{t}{f}_{b}^{*}, \langle\xi\rangle \nabla_{x'}{f}_{b}^{*} \in C(\mathbb R ; L^{2}(\mathbb R^{2}\times\mathbb R^{3})), 
\\
{f}_{b}^{*}, \langle\xi\rangle^{-1} \partial_{t}{f}_{b}^{*}, \nabla_{x'} {f}_{b}^{*} \in L^{\infty}(\mathbb R; L^{\infty}_{\beta}(\mathbb R^{2}\times\mathbb R^{3})),
\end{gather}
\end{subequations}
where $T^{*}$ is a positive constant. For the notational convenience, we use
\begin{align*}
\delta^{*}:=&\tilde{\delta}+\sup_{0 \leq t \leq T^{*} } \left\{ \|(\langle\xi\rangle {f}_{b}^{*},\partial_{t}{f}_{b}^{*},\langle\xi\rangle \nabla_{x'}{f}_{b}^{*} )(t) \|_{L^{2}(\mathbb R^{2}\times\mathbb R^{3})} \right.
\notag \\
&+ \left. \| ({f}_{b}^{*},\langle\xi\rangle^{-1}\partial_{t}{f}_{b}^{*},\nabla_{x'}{f}_{b}^{*})(t) \|_{L^{\infty}_{\beta}(\mathbb R^{2}\times\mathbb R^{3})} \right\}.
\end{align*}
Let us define an extension of the boundary data
\begin{equation*}
U^{*}(t,x,\xi):=\varphi(x_{1})f_{b}^{*}(t,x',\xi).
\end{equation*}
It is straightforward to see that
\begin{subequations}\label{tExBdry0}
\begin{gather}
U^{*}(t,0,x',\xi)=f_{b}^{*}(t,x',\xi),
\label{tExBdry1} \\
U^{*}(t,x_1,x',\xi)=0 \quad t \in \mathbb R, \ x_1 >2, \ x' \in \mathbb R^{2}, \ \xi \in \mathbb R^{3},
\label{tExBdry2} \\
U^{*}(t+T^{*},x,\xi)=U^{*}(t,x,\xi), \quad  t \in \mathbb R, \ x \in {\mathbb R}^{3}_{+} , \   \xi \in \mathbb R^{3},
\label{tExBdry4} \\
\sup_{0 \leq t \leq T^{*} }\left\{\|(\langle\xi\rangle U^{*}, \partial_{t}U^{*}, \langle\xi\rangle \nabla_{x} U^{*})(t)\|  + \|(U^{*},  \langle\xi\rangle^{-1}\partial_{t}U^{*}, \nabla_{x} U^{*})(t)\|_{\beta} \right\}  \lesssim  \delta^*.
\label{tExBdry3}
\end{gather}
\end{subequations}

\begin{theorem}\label{th6}
Let $\beta>7/2$ and $R>0$.  
Assume that the boundary data $f_{b}$ satisfies \eqref{bc_assump2}.
There {exist} constants $\sigma>0$, $\eta_{0}>0$,  and $C_0>0$
such that if $\delta^{*} \leq \eta_{0}$, the time-periodic problem \eqref{tree0} has
a unique solution $f^{*}$ as $e^{\sigma x_{1}}(f^{*}-\varphi_{R}\tilde{f}-U^{*}) \in X(-\infty,0,\beta)$ and
$\tri e^{\sigma x_{1}}(f^{*}-\varphi_{R}\tilde{f}-U^{*}) \tri_{-\infty,0,\beta} \leq C_{0}\delta^{*}$.
\end{theorem}

\begin{theorem}\label{th7}
Let $\beta>7/2$ and $R>0$. 
Assume that the boundary data $f_{b}$ satisfies \eqref{bc_assump2}.
There {exist} constants $\sigma>0$, $\eta_{0}>0$,  and $C_0>0$
such that if $[[e^{\sigma x_{1}}(f_{0}-\varphi_{R}\tilde{f}-U^{*}(0))]]_{\beta}+\delta^{*} \leq \eta_{0}$, 
the initial-boundary value problem \eqref{ree0} has
a unique solution $f$ as $e^{\sigma x_{1}}(f-\varphi_{R}\tilde{f}-U^{*}) \in X(0,0,\beta)$ and
$\tri e^{\sigma x_{1}}(f-\varphi_{R}\tilde{f}-U^{*}) \tri_{0,0,\beta} 
\leq C_{0}([[e^{\sigma x_{1}}(f_{0}-\varphi_{R}\tilde{f}-U^{*}(0))]]_{\beta}+\delta^{*})$.
Moreover, it holds that
\begin{equation*}
[[e^{\sigma x_{1}}(f-f^{*})(t)]]_{\beta} \leq C e^{-\kappa t/2} \quad \text{for $t>0$},
\end{equation*}
where $\kappa$ and $C$ are some positive constants.
\end{theorem}

Note that it is hard to directly solve the stationary problem \eqref{sree0}. 
This is different from the case when one has $\mathbb{R}_+$,
where the stationary solution only depends on $x_1$ and 
therefore the equation \eqref{ree1} reduces to an ODE \eqref{ste1}. 
To get around this difficulty, we borrow the ideas in \cite{ST20,SZ20,Va83},
which discussed the stationary solutions depending all direction $(x_{1},x')$
for the Euler--Poisson equations and the compressible Navier--Stokes equation.
More precisely, we first prove the existence of a time-global solution to the problem \eqref{ree0}
by using the weight function $e^{\sigma x_1}$ introduced in \cite{UYY03,UYY04}.  
By a similar method as in \cite{ST20,SZ20,Va83},
we construct a time-periodic solution making use of this time-global solution. 
If the boundary data $f_{b}$ is time-independent, 
it can be shown by the arbitrariness of period that the time-periodic solution is also time-independent.

Before closing this section, we mention the outline of this paper.
In Section \ref{sec2}, we reformulate the initial-boundary value problem \eqref{boltz0} 
into an initial-boundary value problem for a perturbation 
from the stationary solution $\tilde{f}$ in the half-space, 
as stated in \eqref{pe0}.  
In Section \ref{sec3}, we provide some preliminaries of the Boltzmann equation and the initial-boundary value problem \eqref{pe0}. In particular, we establish crucial estimates on the collision kernel, see {Proposition} \ref{LGamma} and Lemma \ref{K1}.
In Section \ref{sec4}, we consider the linearized problem corresponding to \eqref{pe0} and obtain decay estimates for the linear solution operators. The derivation of the decay estimate involves methods introduced in \cite{UYY04}.
In Section \ref{sec5}, we show the unique existence of the time-global solution to the reformulated problem \eqref{pe0} with small data (see Theorem \ref{global1}) by proving an a priori estimate. 
In Section \ref{sec6}, we consider translated time-global solutions $g_{k}(t,x,\xi):=g(t+kT^*,x,\xi)$ for any $T^* >0$ and $k=0,1,2,\ldots$ to \eqref{pe0} and construct time-periodic and stationary solutions of equation \eqref{pe0}, and in the end we prove {that those solutions are time asymptotic stable.}

\section{Reformulation}\label{sec2}
For the proof of Theorems \ref{th1}--\ref{th7},
we begin by reformulating {{the}} initial-boundary value problem \eqref{ree0}. 
Let us introduce perturbation
\begin{gather*}
g(t,x,\xi):= e^{\sigma x_{1}}(f-\varphi_{R}\tilde{f}-U) (t,x,\xi),
\end{gather*}
where $U=U^{s},U^{*}$ and $\sigma \in (0, \gamma_{0}/2]$ with $ \gamma_{0}$ being in Proposition \ref{ex-st}.

Owing to equations \eqref{ree1}, (\ref{ste1}) and \eqref{ExBdry4}, the perturbation $g$ satisfies the equation
\begin{subequations}\label{pe0}
\begin{gather}
\partial_{t}g + \xi \cdot \nabla_{x} g - \sigma \xi_{1} g - Lg  
= e^{-\sigma x_{1}} \Gamma(g) +  2\Gamma(\varphi_{R}\tilde{f}+U,g)  + e^{\sigma x_{1}}H.
\label{pe1}
\end{gather}
The boundary and initial conditions for $g$  follow from 
(\ref{reic1}), (\ref{rebc1}), (\ref{rebc2}), (\ref{stbc1}), and (\ref{stbc2}) as
\begin{gather}
g(0,x,\xi) = g_{0}(x,\xi) := e^{\sigma x_{1}} (f_{0}(x,\xi) - \varphi_{R} \tilde{f}(x,\xi) -U(0,x,\xi)), \quad x \in {\mathbb R}^{3}_{+} , \   \xi \in \mathbb R^{3},
\label{pic1}\\
g(t,0,x',\xi) = 0, \quad t>0, \  x' \in \mathbb R^{2}, \ \xi_{1} >0,
\label{pbc1}\\
\lim_{x_{1}\to\infty} g (t,x_{1},x',\xi) \to  0 \ \ (x_{1} \to \infty),  \quad t>0, \ x' \in \mathbb R^{2}, \ \xi \in \mathbb R^{3}.
\label{pbc2}
\end{gather}
\end{subequations}
Here the inhomogeneous term $H=H(t,x,\xi)$ is defined by
\begin{equation}\label{defH0}
H:=\left\{\begin{array}{ll}
\Gamma(\varphi_{R}^{1/2}\tilde{f}) -\sum_{i=2}^{3} \xi_{i} (\partial_{x_{i}}\varphi_{R})\tilde{f} - F_{b}\xi_{1} \partial_{x_{1}}  \varphi_{1} + LU + \Gamma(\varphi_{R}\tilde{f}+U) & \text{if $U=U^{s}$},
\\
 \Gamma(\varphi_{R}^{1/2}\tilde{f}) -\sum_{i=2}^{3} \xi_{i} (\partial_{x_{i}}\varphi_{R})\tilde{f} -\partial_{t} U- \xi \cdot \nabla U + LU + \Gamma(\varphi_{R}\tilde{f}+U) & \text{if $U=U^{*}$}.
\end{array}\right.
\end{equation}
Note that $U$ and $H$ are either time-independent or time-periodic.
We also see that
\begin{gather}
[[U]]_{\beta}
\lesssim  
\left\{
\begin{array}{ll}
\delta^{s} \quad \text{if $U=U^{s}$},
\\
\delta^{*} \quad \text{if $U=U^{*}$},
\end{array}
\right.
\label{U1}
\\
[[e^{\sigma x_{1}} H ]]_{{\beta-1}} \lesssim
\left\{
\begin{array}{ll}
\delta^{s} \quad \text{for $\beta>7/2$ if $U=U^{s}$},
\\
\delta^{*} \quad \text{for $\beta>7/2$ if $U=U^{*}$}.
\end{array}
\right.
\label{H1}
\end{gather}
The first inequality follows from \eqref{ExBdry3} and \eqref{tExBdry3}.
The proof of the second inequality is postponed until Lemma \ref{L:H1}.

It suffices to show Theorems \ref{th4} and \ref{th5} below 
for the completion of the proofs of Theorems \ref{th1}--\ref{th7}.
We remark that the stationary (time-independent) problem
is a special case of the time-periodic problem.

\begin{theorem}\label{th4}
Let $\beta>7/2$, $R>0$, and $(U,\delta)=(U^{s},\delta^{s}),(U^{*},\delta^{*})$.
Assume that $U$ and $H$ are time-periodic with a period $T^{*}$ and satisfies \eqref{U1} and \eqref{H1}.
There {exist} constants $\sigma>0$, $\eta_{0}>0$,  and $C_0>0$
such that if $\delta \leq \eta_{0}$, the time-periodic problem corresponding to \eqref{pe0} has
a unique solution $g^{*} \in X(-\infty,0,\beta)$ with
$\tri g^{*} \tri_{-\infty,0,\beta} \leq C_{0}\delta$.
Furthermore, $g^{*}$ is time-independent if $U$ and $H$ are time-independent.
\end{theorem}

\begin{theorem}\label{th5}
Let $\beta>7/2$, $R>0$, and $(U,\delta)=(U^{s},\delta^{s}),(U^{*},\delta^{*})$.
Assume that $U$ and $H$ are time-periodic with a period $T^{*}$ and satisfies \eqref{U1} and \eqref{H1}.
There {exist} constants $\sigma>0$, $\eta_{0}>0$,  and $C_0>0$
such that if $[[g_{0}]]_{\beta}+\delta \leq \eta_{0}$, 
the initial-boundary value problem \eqref{pe0} has a unique solution $g \in X(0,0,\beta)$ with
$\tri g \tri_{0,0,\beta} \leq C_{0}([[g_{0}]]_{\beta}+\delta)$.
Moreover, it holds that
\begin{equation*}
[[(g-g^{*})(t)]]_{\beta} \leq C e^{-\kappa t/2} \quad \text{for $t>0$},
\end{equation*}
where $\kappa$ and $C$ are some positive constants.
\end{theorem}

\section{Preliminaries}\label{sec3}

This section provides the properties of $L$, $\Gamma$, and $H$.
These are summarized in the following proposition, where
$P$ denotes the orthogonal projection onto the null space $N \subset L^{2}_{\xi}$ of $L$. 
The null space $N$ is represented as 
\begin{equation*}
N = {\rm span} \{W^{0}, \ W^{0}\xi_{1}, \ W^{0}\xi_{2}, \ W^{0}\xi_{3}, \ W^{0}|\xi|^{2}\}.
\end{equation*}
For more details, see (1.5) and (1.19) in \cite{UYY03}.

\begin{proposition}\label{LGamma}
The following hold with some positive constants
$\nu_{0}$, $\nu_{1}$, $\nu_{2}$, $k_{0}$, $k_{1}$, $k_{2}$, $k_{3}$, $k_{4}$, $k_{5}$, and $l_{0}$.
Here $l_{0}$ depends only on $\sigma_{0}$, $\rho_{\infty}$, $u_{\infty}$, $T_{\infty}$, and $\beta$.
The other constants depend only on $\sigma_{0}$, $\rho_{\infty}$, $u_{\infty}$, and $T_{\infty}$.
\begin{enumerate}[(i)]
\item $L$ has the decomposition
\begin{equation*}
Lf=-\nu(\xi) f +Kf, \quad Kf:=\int_{{\mathbb R}^{3}} K(\xi,\xi') h(\xi') d\xi',
\end{equation*}
where $\nu$ is a positive function with
\begin{equation*}
\nu_{0} \langle \xi \rangle \leq \nu(\xi) \leq \nu_{0}^{-1} \langle \xi \rangle , 
\end{equation*}
and the kernel $K(\xi,\xi')$ satisfies
\begin{equation*}
|K(\xi,\xi')| \leq k_{0}(|\xi-\xi'|+|\xi-\xi'|^{-1}) e^{-k_{1}|\xi-\xi'|}.
\end{equation*}
\item $L$ is non-positive self-adjoint on $L^{2}_{\xi}$ with 
\begin{equation*}
\int_{\mathbb R^{3}} f Lf d\xi \leq -\nu_{1} \| \langle \xi \rangle^{1/2} (I - P) f \|_{L^{2}_{\xi}}^{2}
 \quad \text{for $f \in L^{2}_{\xi}$}.
\end{equation*}
\item $K$ has the regularizing properties, that is, the following maps are bounded.
\begin{equation*}
K: L^{\infty}_{\xi,\beta} \to L^{\infty}_{\xi,\beta+1}, \quad 
K: L^{2}_{\xi} \to L^{2}_{\xi} \cap L^{\infty}_{\xi}.
\end{equation*}
 \item The bilinear operator $\Gamma(f,g)$ satisfies $\Gamma(f,g)=\Gamma(g,f)$,
\begin{align*}
\|\nu^{-1} \Gamma (f,g) \|_{L^{\infty}_{\xi,\beta}} 
&\leq k_{3} \|f\|_{L^{\infty}_{\xi,\beta}}\|g\|_{L^{\infty}_{\xi,\beta}} \quad \text{for $\beta>0$}, \quad
\\
\|\Gamma(f,g)\| 
&\leq l_{0} (\|f\|_{\beta}\|g\| +\|f\| \|g\|_{\beta})
\quad \text{for $\beta > 7/2$},
\\
\|\Gamma(f,g)h\|_{L^{1} } 
&\leq k_{4} \|\langle \xi \rangle^{1/2}f\|_{L^{\infty}_{x}L^{2}_{\xi}}\|\langle \xi \rangle^{1/2}g\| \|\langle \xi \rangle^{1/2}h\| 
\end{align*}
for every functions $f$, $g$, and $h$ with the relevant norm bounded.
\item Under condition \eqref{super1}, $A=P\xi_{1}P$ is negative definite on the null space $N \subset L^{2}_{\xi}$, that is, 
\begin{equation*}
(A f,f)_{L^{2}_{\xi}} \leq -\nu_{2} \|f\|_{L^{2}_{\xi}} \quad \text{for $f \in N$}.
\end{equation*}
\item $P\xi_{1}$ is bounded from $L^2_\xi$ to $L^2_\xi$, that is, 
\begin{equation*}
\Vert P\xi_1 f\Vert_{L^2_\xi}\le k_{5} \Vert  f\Vert_{L^2_\xi}  \quad \text{for $f \in L^{2}_{\xi}$}.
\end{equation*}
\end{enumerate}
\end{proposition}

\begin{proof}
The property (i) except the lower bound of $\nu$ is shown in Proposition 2.1 in \cite{UYY04}. 
From the definition of $\nu$ in the proof of Proposition 2.1 in \cite{UYY04} 
and the inequality $c \langle \xi \rangle \leq \nu^{0}(\xi) \leq C \langle \xi \rangle $ shown in Chapter 3 of \cite{glassey}, 
one can also obtain the lower bound {{in (i)}}.

Proposition 2.1 in \cite{UYY04} also provides properties (ii)--(iv) except the second and third inequality in (iv). 
(See also Chapter 3 in \cite{glassey}.) Property (v) is proved just after (2.5) in \cite{UYY03}.
Furthermore, property (vi) follows from the fact that
$P$ is a projection to a space spanned by the Maxwellian multiplied with polynomials.

What is left is to show the second and third inequalities {{in}} (iv). Set $M(\xi)=M[1,u_{\infty},T_{\infty}](\xi)$.
{{Firstly}} let us show the second one.
A typical term of the loss terms of $\Gamma$ is estimated by the Schwarz inequality as
\begin{align}
\vert \Gamma_{loss}(x,\xi)\vert&:= \Big\vert \int_{\mathbb{R}^3}\!\!\int_{\bm{S}^{2}}  |(\xi-\xi_*)\cdot \omega \vert  f(x,\xi) {M}^{1/2}(\xi_{*})  g(x,\xi_{*}) d\xi_{*} d\omega \Big\vert
\notag\\
&\lesssim \langle \xi \rangle |f(x,\xi)| \int_{\mathbb{R}^3}\!\!\int_{\bm{S}^{2}}  |\xi_*| {M}^{1/2}(\xi_{*})|g(x,\xi_{*})| d\xi_{*} d\omega
\label{Gamma1}\\
&\lesssim \langle \xi \rangle^{1-\beta} \|f\|_{\beta} \|g(x)\|_{L^{2}_{\xi}}.
\notag
\end{align}
For $\beta > 5/2$,  we have
\begin{equation*}
\|\Gamma_{loss}\|^{2} \lesssim \|f\|_{\beta}^{2} \|g\|^{2}.
\end{equation*}
A typical part of the gain terms is handled by the Schwarz inequality as
\begin{align}
\Gamma_{gain}^{2}(x,\xi) & := \left({M}^{-1/2}(\xi)
\int_{\mathbb{R}^3}\!\!\int_{\bm{S}^{2}}  |(\xi-\xi_*)\cdot \omega \vert 
{M}^{1/2}(\xi')  f(x,\xi') {M}^{1/2}(\xi_{*}')  g(x,\xi_{*}') d\xi_{*} d\omega \right)^{2}
\notag \\
& \lesssim \left( \int_{\mathbb{R}^3}\!\!\int_{\bm{S}^{2}}  |(\xi-\xi_*)\cdot \omega \vert 
{M}^{1/2}(\xi_{*})  |f(x,\xi') g(x,\xi_{*}')| d\xi_{*} d\omega \right)^{2}
\notag \\
& \lesssim \left(\int_{\mathbb{R}^3}\!\!\int_{\bm{S}^{2}}  |(\xi-\xi_*)\cdot \omega \vert {M}^{1/2}(\xi_{*})  d\xi_{*} d\omega \right)
\notag \\
& \qquad \times \left(\int_{\mathbb{R}^3}\!\!\int_{\bm{S}^{2}}  |(\xi-\xi_*)\cdot \omega \vert {M}^{1/2}(\xi_{*})  f^{2}(x,\xi') g^{2}(x,\xi_{*}') d\xi_{*} d\omega \right)
\notag \\
& \lesssim \langle \xi \rangle^{2} \int_{\mathbb{R}^3}\!\!\int_{\bm{S}^{2}} f^{2}(x,\xi') g^{2}(x,\xi_{*}') d\xi_{*} d\omega
\label{Gamma2} \\
& \lesssim \langle \xi \rangle^{2} \int_{\mathbb{R}^3}\!\!\int_{\bm{S}^{2}} |f(x,\xi') g(x,\xi_{*}')| \|f\|_{\beta}\|g\|_{\beta} \langle \xi' \rangle^{-\beta}\langle \xi_{*}' \rangle^{-\beta} d\xi_{*} d\omega
\notag \\
& \lesssim \|f\|_{\beta}\|g\|_{\beta} \int_{\mathbb{R}^3}\!\!\int_{\bm{S}^{2}} |f(x,\xi') g(x,\xi_{*}')| \langle \xi' \rangle^{2-\beta}\langle \xi_{*}' \rangle^{2-\beta}  d\xi_{*} d\omega,
\notag
\end{align}
where we have used $\langle \xi \rangle^{2} \leq \langle \xi' \rangle^2\langle \xi_{*}' \rangle^{2}$ in deriving the last inequality.
For $\beta > 7/2$, we have
\begin{align*}
\|\Gamma_{gain}\|^{2}
&\lesssim \|f\|_{\beta}\|g\|_{\beta} \int_{\mathbb R^{3}_{+}} \left( \int_{\mathbb R^{3}}  \int_{\mathbb{R}^3}\!\!\int_{\bm{S}^{2}} |f(x,\xi') g(x,\xi_{*}')|  \langle \xi' \rangle^{2-\beta}\langle \xi_{*}' \rangle^{2-\beta} d\xi_{*} d\omega d\xi  \right) dx
\\
&= \|f\|_{\beta}\|g\|_{\beta} \int_{\mathbb R^{3}_{+}} \left( \int_{\mathbb R^{3}}  \int_{\mathbb{R}^3}\!\!\int_{\bm{S}^{2}} |f(x,\xi) g(x,\xi_{*})|  \langle \xi \rangle^{2-\beta}\langle \xi_{*} \rangle^{2-\beta}
d\xi_{*} d\omega d\xi  \right) dx
\\
& \lesssim  \|f\|_{\beta}\|g\|_{\beta} \int_{\mathbb R^{3}_{+}}  \|f(x)\|_{L^{2}_{\xi}} \|g(x)\|_{L^{2}_{\xi}} dx
\\
& \lesssim  \|f\|_{\beta}\|g\|_{\beta} \|f\| \|g\| ,
\end{align*}
where we have used the change of variables $(\xi,\xi_{*},\omega) \to (\xi',\xi_{*}',-\omega)$ in deriving the equality;
we have used the Schwarz inequality in deriving the last two inequalities.
Combining the estimates above leads to the second inequality in (iv).

Next let us show the third inequality in (iv).
Using \eqref{Gamma1}, we estimate the loss term of $\Gamma$ as
\begin{align*}
\|\Gamma_{loss} h\|_{L^{1}} 
&\lesssim  \int_{\mathbb R^{3}_{+}}  \|\langle \xi \rangle^{1/2}f(x)\|_{L^{2}_{\xi}} \|g(x)\|_{L^{2}_{\xi}} \|\langle \xi \rangle^{1/2}h(x)\|_{L^{2}_{\xi}} dx
\\
& 
\lesssim
\|\langle \xi \rangle^{1/2}f\|_{L^{\infty}_{x}L^{2}_{\xi}}\|g\| \|\langle \xi \rangle^{1/2}h\| .
\end{align*}
By \eqref{Gamma2} and $\langle \xi \rangle^{2} \leq \langle \xi' \rangle^2\langle \xi_{*}' \rangle^{2}$, 
the gain term is handled as
\begin{align*}
\Gamma_{gain}^{2}(x,\xi) 
& \lesssim \langle \xi \rangle \int_{\mathbb{R}^3}\!\!\int_{\bm{S}^{2}} |f(x,\xi') g(x,\xi_{*}')|^{2} \langle \xi' \rangle \langle \xi_{*}' \rangle  d\xi_{*} d\omega.
\end{align*}
Then we have
\begin{align*}
\|\Gamma_{gain}h\|_{L^{1}}^{2}
&\lesssim \| \langle \xi \rangle^{-1/2} \Gamma_{gain} \|^{2} \|\langle \xi \rangle^{1/2}h\|^{2}
\\
&\lesssim  \int_{\mathbb R^{3}_{+}} \left( \int_{\mathbb R^{3}}  \int_{\mathbb{R}^3}\!\!\int_{\bm{S}^{2}} |f(x,\xi') g(x,\xi_{*}')|^{2} \langle \xi' \rangle \langle \xi_{*}' \rangle d\xi_{*} d\omega d\xi  \right) dx \|\langle \xi \rangle^{1/2}h\|^{2}
\\
&= \int_{\mathbb R^{3}_{+}} \left( \int_{\mathbb R^{3}}  \int_{\mathbb{R}^3}\!\!\int_{\bm{S}^{2}} |f(x,\xi) g(x,\xi_{*})|^{2} \langle \xi \rangle \langle \xi_{*} \rangle d\xi_{*} d\omega d\xi  \right) dx \|\langle \xi \rangle^{1/2}h\|^{2} 
\\
&\lesssim \int_{\mathbb R^{3}_{+}}  \|\langle \xi \rangle^{1/2}f(x)\|_{L^{2}_{\xi}}^{2} \|\langle \xi \rangle^{1/2}g(x)\|_{L^{2}_{\xi}}^{2} dx  \|\langle \xi \rangle^{1/2}h\|^{2}
\\
& \lesssim \|\langle \xi \rangle^{1/2}f\|_{L^{\infty}_{x}L^{2}_{\xi}}^{2} \|\langle \xi \rangle^{1/2}g\|^{2} \|\langle \xi \rangle^{1/2}h\|^{2},
\end{align*}
where we have used the change of variables $(\xi,\xi_{*},\omega) \to (\xi',\xi_{*}',-\omega)$ in deriving the equality.
Combining the estimates above leads to the third inequality in (iv).
\end{proof}


{{In the next lemma, we give some estimates on the kernel $K(\xi,\xi')$.}}

\begin{lemma}\label{K1}
There hold that
\begin{gather}
\sup_{\xi\in \mathbb{R}^3}\int_{\mathbb{R}^3} \vert K(\xi, \xi')\vert^\alpha d\xi'<C(\alpha) \quad \text{for $0<\alpha<3$},
\label{K2}
\\
\int_{\mathbb{R}^3}\Big(\int_{\mathbb{R}^3} \vert K(\xi, \xi')\vert^\alpha d\xi'\Big)^\beta d\xi<C({\alpha,\beta})
\quad \text{for $0<\alpha<3$, $\beta>3$}.
\label{K3}
\end{gather}
Furthermore, there exists a pair of $p \in (3,4)$ and $q \in (1,3)$ 
such that $p<2q$ and
\begin{align}\label{K4}
A_{p',q'}:=\Big(\int_{\mathbb{R}^3} \Big(\int_{\mathbb{R}^3} \vert K(\xi, \xi')\vert^{p'/2} d\xi'\Big)^{2q'/p'} d\xi\Big)^{1/q'}<+\infty,
\end{align}
where $p'$ and $q'$ are the H\"{o}lder conjugates of $p$ and $q$, respectively.
\end{lemma}
\begin{proof}
It is straightforward to see from Proposition \ref{LGamma} (i) that
\begin{align*}
\int_{\mathbb{R}^3} \vert K(\xi, \xi')\vert^\alpha d\xi'\le C(\alpha) \langle \xi\rangle^{-1}.
\end{align*}
This immediately yields \eqref{K2} and \eqref{K3}.

By \eqref{K2} and \eqref{K3}, the conditions $0<\frac{p'}{2}<3$ and $3<\frac{2q'}{p'}$ ensure $A_{p',q'} < +\infty$.
Therefore, it suffices to find a pair $p$ and $q$ with
\begin{align}
\begin{cases}
3<p<4,\\
1<q<3, \\
p<2q,\\
0<\frac{p'}{2}<3, \\ 
3<\frac{2q'}{p'}. \\ 
\end{cases}
\Longleftrightarrow
\quad
\begin{cases}
3<p<4,\\
1<q<3,\\
p<2q,\\
pq<3p-2q.
\end{cases}
\Longleftrightarrow
\quad
\begin{cases}
\frac14<\frac1p<\frac13,\\
\frac13<\frac1q<1,\\
\frac1q<\frac2p,\\
1<\frac3q-\frac2p.
\end{cases}
\label{pq}
\end{align}
It is easy to find some pairs of $p$ and $q$ satisfying the rightmost in \eqref{pq}.
The proof is complete.
\end{proof}

Using \eqref{ExBdry0}, \eqref{tExBdry0}, and Propositions \ref{ex-st} and \ref{LGamma}, 
we have the following lemma. We omit the proof since it is straightforward.

\begin{lemma} \label{L:H1}
The quantity $H$ defined in \eqref{defH0} can be estimated as \eqref{H1}.
\end{lemma}

\section{Linearized problem}\label{sec4}
We consider the following linearized problem:
\begin{subequations}\label{le0}
\begin{gather}
\partial_{t}h + \xi \cdot \nabla_{x} h - \sigma \xi_{1} h - Lh  = 0,
\label{le1} \\
h(0,x,\xi) = h_{0}(x,\xi), \quad x \in {\mathbb R}^{3}_{+} , \   \xi \in \mathbb R^{3},
\label{lic1}\\
h(t,0,x',\xi) = 0, \quad t>0, \  x' \in \mathbb R^{2}, \ \xi_{1} >0,
\label{lbc1}\\
\lim_{x_{1}\to\infty} h (t,x_{1},x',\xi) \to  0 \ \ (x_{1} \to \infty),  \quad t>0, \ x' \in \mathbb R^{2}, \ \xi \in \mathbb R^{3}.
\label{lbc2}
\end{gather}
\end{subequations}
We denote the solution operator of this problem by $S(t)$. 
For the solvability of  the linearized problem, the following lemma holds by the standard method:

\begin{lemma}
For any $T>0$ and $h_{0} \in L^{2}$,  
problem \eqref{le0} has a unique solution $h \in C([0,T];L^{2})$.
Furthermore, the solution operator $S(t)$ is a $C_{0}$-semigroup  on $L^{2}$.
\end{lemma}

Let us also consider an initial boundary value problem of a damped transport equation:
\begin{subequations}\label{de0}
\begin{gather}
\partial_{t} h + \xi \cdot \nabla_{x} h - \sigma \xi_{1} h - \nu(\xi) h  = 0,
\label{de1} \\
h(0,x,\xi) = h_{0}(x,\xi), \quad x \in {\mathbb R}^{3}_{+} , \   \xi \in \mathbb R^{3},
\label{dic1}\\
h(t,0,x',\xi) = 0, \quad t>0, \  x' \in \mathbb R^{2}, \ \xi_{1} >0,
\label{dbc1}\\
\lim_{x_{1}\to\infty} h (t,x_{1},x',\xi) \to  0 \ \ (x_{1} \to \infty),  \quad t>0, \ x' \in \mathbb R^{2}, \ \xi \in \mathbb R^{3}.
\label{dbc2}
\end{gather}
\end{subequations}
We denote the solution operator of problem \eqref{de0} by $S_{0}(t)$, and know that
\begin{equation}\label{form1}
(S_{0} (t) h_0) (t, x, \xi) = e^{-(\nu (\xi)- \sigma \xi_1) t} \chi (x_1  - \xi_1 t)  h_0 (x_1 - \xi_1 t, x_2- \xi_2 t, x_3-\xi_3 t, \xi)
\end{equation}
with $\chi (s)$ being the 1D characteristic function of the set $\{ s > 0 \}$.
It is obvious that $S_{0}(t)$ is  a $C_{0}$-semigroup on $L^{2}$.
Note that the solution operators $S(t)$ and $S_{0}(t)$ has the relation 
\begin{equation}\label{SS0}
S(t)h_{0}=S_{0}(t)h_{0}+S_{0}*(KSh_{0}), \quad \text{where} \ S_{0}*f:=\int_{0}^{t} S_{0}(t-s)f(s)ds.
\end{equation}

\subsection{Decay estimates in $L^{2}$ and $L^{\infty}_{\beta}$}\label{sec4.1}
This subsection deals with the decay estimate of $S(t)h_{0}$
in the function spaces $L^{2}$ and $L^{\infty}_{\beta}$.
In much the same way as in the proof of (2.10) in \cite{UYY03},
one can show the following lemma on the decay estimate in $L^{2}$.

\begin{lemma}\label{lem2}
For any $h_{0} \in {L^{2} }$, the solution operator $S(t)$ satisfies
\begin{gather*}
\| S(t)h_{0} \|  \leq  \|h_{0} \| e^{-\kappa t},
\end{gather*}
where $\kappa \leq \nu_{0}/2$ is a positive constant depending only on $\nu_{0}$, $\nu_{1}$, and $\nu_{2}$
being in Proposition \ref{LGamma}. 
\end{lemma}

We derive the decay estimate of $S(t)h_{0}$ in $L^{\infty}_{\beta}$
by using the solution operator $S_{0}(t)$ for the problem \eqref{de0}.
It is straightforward to see from \eqref{form1} and Proposition \ref{LGamma} (i) that
\begin{equation}\label{S0decay}
\|S_{0} (t) h_0 \|_{X} \lesssim  \| h_{0}\|_{X} e^{-(2\kappa-\epsilon)t} \quad 
\text{for $X=L^{2} $, $L^{\infty}_{\beta}$}
\end{equation}
with $\kappa$ being in Proposition \ref{LGamma} (i) and some positive constant $\epsilon \ll \kappa$.
From \eqref{SS0}, we have 
\begin{align}
S(t) h_0 &= S_0 (t) h_0 + \int^t_0 S_0 (t-s) KS (s) h_0 ds = \sum_{j=0}^{m-1} I_j (t) + J_m (t),  
\label{sg-h-1}
\end{align}
where
\begin{align*}
I_0 (t) &:= S_0 (t) h_0 ,
\\
I_j (t) &:= \int^t_0 S_0 (t-s) KI_{j-1} (s) ds = (S_0 K) * I_{j-1} ,
\\
J_m (t) &:= (S_0 K) *(S_0 K) * \cdots *(S_0 K) * h , \quad h(t) = S(t) h_0.
\end{align*}
Here $*$ means convolution in $t$,
and there are $m$ copies of $S_0 K$ in $J_{m}$.

\begin{lemma}\label{lemS}
Let $\beta>3/2$. For $h_{0} \in  L^{2}  \cap L^{\infty}_{\beta}$, there holds that
\begin{gather*}
\| S(t) h_{0} \|_{\beta}  \lesssim  [[ h_{0} ]]_{\beta} e^{-\kappa t}.
\end{gather*}
\end{lemma}
\begin{proof}
We can estimate the term $I_{j}$ in \eqref{sg-h-1} by using \eqref{S0decay} as 
\begin{align}\label{I_j}
\|I_j(t) \|_\beta \lesssim \|h_0 \|_{\max\{0,\beta-j\}} e^{-(2\kappa - \epsilon)t} , 
\end{align}
where we have retaken $\epsilon$ so that both \eqref{S0decay} and \eqref{I_j} hold with the same $\epsilon$.
For the completion of the proof, it suffices to show that for any $\epsilon_{1}>0$,
\begin{equation} \label{J_beta1}
\|J_{[\beta] +4} (t) \|_\beta 
\leq \epsilon_{1}e^{-\kappa t}  \sup_{0\leq \tau \leq t} e^{\kappa t} \| S(\tau) h_{0} \|_{\beta} 
+ C({\epsilon_{1}}) \| h_{0} \| e^{-\kappa t},
\end{equation}
where $[\cdot]$ is the floor function.
Indeed, letting $m=[\beta]+4$, taking the $L^{\infty}_{\beta}$ norm of \eqref{sg-h-1}, 
multiplying it by $e^{\kappa t}$, taking the supremum of the result in $t$, 
and letting $\epsilon_{1}$ be small, 
we have the desired decay estimate.

To show \eqref{J_beta1}, we set $m:=[\beta]$ and observe by using \eqref{S0decay} that 
\begin{equation}\label{J_beta2}
\| J_{m+4} (t) \|_\beta \lesssim \frac{1}{(m+1)!} \int^t_0 (t-\tau)^{m+1} e^{-(2\kappa - \epsilon)(t-\tau)} \|J_2 (\tau)\|_{L^\infty_x L^2_\xi} d\tau.
\end{equation}
Recalling $h =  S(t)h_{0}$ and $J_2 (t) =  (S_0 K) * (S_0 K) * h$, we can write
\begin{equation*}
J_2 (t) =  S_0 * \bar{J}, 
\end{equation*}
where 
\begin{align*}
\bar{J}(t) := & (K S_0 K) * h = \int^t_0 K S_0 (t-s) K h(s) ds = \int^t_0 \bar{J}_0 (t-s, s) ds,
\\
\bar{J}_0 (t, s) := & KS_{0}(t)K h(s)
=  \int_{\mathbb{R}^3}\!\!\int_{\mathbb{R}^3} K(\xi, \xi') K(\xi', \xi'') e^{-(\nu (\xi')- \sigma \xi'_1) t} \chi (y_1) h (s, y, \xi'') d\xi' d\xi''
\end{align*}
with $y= x  - \xi' t$ and $y_1=x_{1}-\xi'_1 t$.
Now we estimate $\bar{J}_0$. From Proposition \ref{LGamma} (i), we have
\begin{align*}
\vert \bar{J}_0(t,s) \vert 
& \le e^{-(2\kappa-\varepsilon)t} \int_{\mathbb{R}^3}\!\!\int_{\mathbb{R}^3} \vert K(\xi, \xi') K(\xi', \xi'')  \chi(y_1)  h(s,y,\xi'') \vert d\xi' d\xi'' 
\\
& \le e^{-(2\kappa-\varepsilon)t} \int_{\mathbb{R}^3} \vert K(\xi,\xi') \vert \Big(\int_{\mathbb{R}^3}\vert K(\xi', \xi'')\vert^2d\xi''\Big)^{1/2}
\Big(\int_{\mathbb{R}^3}  \chi(y_1)^2 \vert h(s,y,\xi'')\vert^2d\xi''\Big)^{1/2} d\xi',
\end{align*}
where we have also applied the Schwarz inequality in deriving the last inequality.
Using Lemma \ref{K1}, we have
\begin{align*}
\int_{\mathbb{R}^3} \vert \bar{J}_0(t,s)\vert^2 d\xi 
&\lesssim e^{-2(2\kappa -\varepsilon)t} \int_{\mathbb{R}^3} \Big[ \int_{\mathbb{R}^3} \vert K(\xi,\xi')\vert \chi(y_1) \Vert h(s,y,\cdot)\Vert_{L^2_\xi} d\xi'\Big]^2d\xi.
\end{align*} 
Applying the H\"{o}lder inequality twice leads to
\begin{align*}
\int_{\mathbb{R}^3} \vert \bar{J}_0(t,s)\vert^2 d\xi 
& \lesssim e^{-2(2\kappa -\varepsilon)t} \int_{\mathbb{R}^3}\Big(\int_{\mathbb{R}^3} \vert K(\xi,\xi')\vert^{p'/2}d\xi'\Big)^{2/p'} 
\\
& \qquad \times \Big(\int_{\mathbb{R}^3} \vert K(\xi, \xi')\vert^{p/2} \chi^{p}(y_1) \Vert h(s,y,\cdot)\Vert_{L^2_\xi}^p d\xi'\Big)^{2/p} d\xi
\\
&\lesssim e^{-2(2\kappa -\varepsilon)t} A_{p',q'}
 \Big( \int_{\mathbb{R}^3} \Big( \int_{\mathbb{R}^3} \vert K(\xi, \xi') \vert^{p/2} \chi^{p}(y_1) \Vert h(s,y,\cdot)\Vert_{L^2_\xi}^p d\xi'\Big)^{2q/p} d\xi\Big)^{1/q},
\end{align*}
where $p'$ and $q'$ are the H\"{o}lder conjugates of $p\ge 1$ and $q\ge 1$, respectively, and 
\begin{align*}
A_{p',q'}=\Big(\int_{\mathbb{R}^3} \Big(\int_{\mathbb{R}^3} \vert K(\xi, \xi')\vert^{p'/2} d\xi'\Big)^{2q'/p'} d\xi\Big)^{1/q'}.
\end{align*}
We can choose a pair of $p$ and $q$ so that $p \in (3,4)$, $q \in (1,3)$, $p<2q$, and $A_{p',q'}<+\infty$
thanks to Lemma \ref{K1}.
Then using the Minkowski integral inequality, that is,
\begin{align*}
\Vert \Vert F(\xi,\xi')\Vert_{L^p_{\xi'}}^2\Vert_{L^q_\xi} 
= \Vert \Vert F(\xi,\xi')\Vert_{L^p_{\xi'}} \Vert_{L^{2q}_\xi}^2
\le \Vert \Vert F(\xi,\xi')\Vert_{L^{2q}_\xi} \Vert_{L^p_{\xi'}}^2,
\end{align*}
and Lemma \ref{K1}, we obtain
\begin{align*}
\int_{\mathbb{R}^3} \vert \bar{J}_0(t,s)\vert^2 d\xi 
&\lesssim e^{-2(2\kappa -\varepsilon)t} \Big( \int_{\mathbb{R}^3} \Big( \int_{\mathbb{R}^3} \vert K(\xi, \xi') \vert^{q} \chi^{2q}(y_1) \Vert h(s,y,\cdot) \Vert^{2q}_{L^{2}_{\xi}} d\xi \Big)^{p/2q} d\xi'\Big)^{2/p}
\\
& \lesssim e^{-2(2\kappa -\varepsilon)t} \Big( \int_{\mathbb{R}^3} \chi^{p}(y_1)  \Vert h(s,y,\cdot) \Vert^{p}_{L^{2}_{\xi}}  d\xi'\Big)^{2/p}
\\
& = e^{-2(2\kappa -\varepsilon)t}  \Big( \int_{\mathbb{R}^3_{+}} \Vert h(s,y,\cdot)\Vert^p_{L^{2}_{\xi}} t^{-3}dy\Big)^{2/p}
\\
& = e^{-2(2\kappa -\varepsilon)t}  t^{-6/p}\Big( \int_{\mathbb{R}^3_{+}} \Vert h(s,y,\cdot)\Vert^{p-2}_{L^{2}_{\xi}} \Vert h(s,y,\cdot)\Vert^2_{L^{2}_{\xi}} dy\Big)^{2/p},
\end{align*}
where we have used the change of variables $y= x  - \xi' t$ in deriving the first equality.
Then using $\|\cdot\|_{L^{2}_{\xi}} \lesssim \|\cdot\|_{\xi,\beta}$ for $\beta>3/2$ and Lemma \ref{lem2}, we arrive at
\begin{align*}
\int_{\mathbb{R}^3} \vert \bar{J}_0(t,s)\vert^2 d\xi 
& \lesssim e^{-2(2\kappa -\varepsilon)t}  t^{-6/p} \|h(s)\|_{\beta}^{(2p-4)/p} \Big( \int_{\mathbb{R}^3_{+}}  \Vert h(s,y,\cdot)\Vert^2_{L^{2}_{\xi}} dy\Big)^{2/p}
\\
& \lesssim e^{-2(2\kappa -\varepsilon)t}  t^{-6/p} e^{-2\kappa s} \left(\sup_{0\leq \tau \leq s} e^{\kappa \tau} \| h(\tau) \|_{\beta} \right)^{(2p-4)/p} \|h_{0}\|^{4/p}.
\end{align*}
From this and the fact $p\in (3,4)$, there holds that
\begin{align*}
\|\bar{J}(t)\|_{L^{\infty}_{x}L^{2}_{\xi}}
& \leq \int_{0}^{t} \|\bar{J}_{0}(t-s,s) \|_{L^{\infty}_{x}L^{2}_{\xi}} ds
\\
& \lesssim \int_{0}^{t}  e^{-(2\kappa -\varepsilon)(t-s)}  (t-s)^{-3/p} e^{-\kappa s} \left(\sup_{0\leq \tau \leq s} e^{\kappa \tau} \| h(\tau) \|_{\beta} \right)^{(p-2)/p} \|h_{0}\|^{2/p} ds
\\
& \lesssim e^{-\kappa t}  \left(\sup_{0\leq \tau \leq t} e^{\kappa \tau} \| h(\tau) \|_{\beta} \right)^{(p-2)/p} \|h_{0}\|^{2/p}.
\end{align*}
This together with \eqref{S0decay} implies that 
\begin{align*}
\|J_{2}(t)\|_{L^{\infty}_{x}L^{2}_{\xi}}
& \lesssim e^{-\kappa t}\left(\sup_{0\leq \tau \leq t} e^{\kappa \tau} \| h(\tau) \|_{\beta} \right)^{(p-2)/p} \|h_{0}\|^{2/p}.
\end{align*}
Plugging this into \eqref{J_beta2} and then using the Young inequality, we conclude \eqref{J_beta1}.
The proof is complete.
\end{proof}

\subsection{Estimates of the convolutions in time}\label{sec4.3}

We also estimate the convolutions
\begin{equation*}
S*h:=\int_{0}^{t} S(t-s) h(s) ds, \quad 
S*_{t_{0}}h:=\int_{t_{0}}^{t} S(t-s) h(s) ds.
\end{equation*}

\begin{lemma} \label{lemS*}
Let $\beta >3/2$. There hold that
\begin{gather}
[[ (S *_{t_{0}}  h)(t) ]]_\beta  \leq C_{0}  e^{-\kappa t/2} \sup_{t_{0} \leq \tau \leq t} e^{\kappa \tau /2} [[h (\tau) ]]_{\beta-1},
\label{S*_beta} \\
[[ (S * h)(t) ]]_\beta  \lesssim  \sup_{0 \leq \tau \leq t} [[h (\tau) ]]_{\beta-1}
\label{S*_beta2} 
\end{gather}
for every function $h(t, x, \xi)$ with the relevant norm bounded, 
where $[[\,\cdot \,]]_\beta = \|\cdot\|_\beta + \|\cdot\|$, and $C_{0}$ is a positive constant independent of $t_{0}$.
\end{lemma}
\begin{proof}
We show only \eqref{S*_beta}, since \eqref{S*_beta2} follows from taking $t_{0}=0$ and $\kappa=0$ in \eqref{S*_beta}. 
To this end, we first estimate the norm $\|\cdot\|_{\beta}$ by using
the relation \eqref{SS0}, i.e., $S=S_{0}+S_{0}*KS$.
It is {{observed}} from \eqref{form1} that
\begin{align} \label{S_0*}
\| (S_0 *_{t_{0}} h) (t)\|_\beta
& =\| (S_0 *_{t_{0}} \nu \nu^{-1} h) (t)\|_\beta
\notag \\
& \leq \sup_{x \in \mathbb R^{3}_{+}, \xi \in \mathbb R^{3}} \int^t_{t_{0}} \langle \xi \rangle^\beta  e^{-(\nu (\xi)  -\sigma \xi_1) (t-s)} \chi (x_1 - \xi_1 s) \nu(\xi)\nu^{-1}(\xi) |h(s, x-\xi s, \xi)| ds 
\notag\\
& \leq e^{-\kappa t}  (\sup_{t_{0} \leq \tau \leq t}  e^{\kappa \tau} \|\nu^{-1} h (\tau) \|_\beta ) \sup_{\xi \in \mathbb R^{3}}\left\{ \int^t_{t_{0}} e^{-(\nu (\xi) - \kappa -\sigma \xi_1) (t-s)} \nu(\xi) ds \right\} 
\notag\\
& \leq C_{0} e^{-\kappa t} \sup_{t_{0} \leq \tau \leq t}  e^{\kappa \tau} \|h (\tau)\|_{\beta-1},
\end{align}
where we have also used Proposition \ref{LGamma} (i) in deriving the last inequality.
Now we estimate the other part $S_{0}*KS$. It holds that
\begin{align*}
((S_{0}*KS)*_{t_{0}}h)(t)
&= \int_{t_{0}}^{t} \left\{ \int_{0}^{t-\tau} S_{0}(t-\tau-s) KS(s) ds \right\} h(\tau) d\tau
\\
&= \int_{t_{0}}^{t} S_{0}(t-s) \left\{ \int_{t_{0}}^{s} KS(s-\tau) h(\tau) d\tau \right\}  ds
\\
&=S_{0}*_{t_{0}}(KS*_{t_{0}}h)(t).
\end{align*}
Then using \eqref{S_0*}, Proposition \ref{LGamma} (iii),  and Lemma \ref{lemS}, we have
\begin{align*}
\| S_0 *_{t_{0}} (KS *_{t_{0}} h)(t) \|_\beta
& \leq C_{0}\int_{t_{0}}^{t} e^{-\kappa(t-s)}\|(S*_{t_{0}} h)(s)\|_{\beta-1} ds
\notag \\
& \leq C_{0} \int_{t_{0}}^{t} e^{-\kappa(t-s)} \int_{t_{0}}^{s} e^{-\kappa(s-\tau)}[[h(\tau)]]_{\beta-1} d\tau ds
\notag \\
& \leq C_{0} (\sup_{t_{0} \leq \tau \leq t} e^{\kappa \tau/2} [[h(\tau)]]_{\beta-1}) \int_{t_{0}}^{t} e^{-\kappa(t-s)} \int_{t_{0}}^{s} e^{-\kappa(s-\tau)} e^{-\kappa \tau/2}d\tau ds
\notag \\
& \leq C_{0} e^{-\kappa t/2} \sup_{t_{0} \leq \tau \leq t} e^{\kappa \tau/2} [[h(\tau)]]_{\beta-1}.
\end{align*}

Let us handle the norm $\|\cdot\|$.
Using Lemma \ref{lem2}, we observe that
\begin{align*}
\| (S *_{t_{0}} h)(t) \|
& \leq \int_{t_{0}}^{t} e^{-\kappa(t-s)}\|h(s)\| ds
\notag \\
& \leq \sup_{t_{0} \leq \tau \leq t}  e^{\kappa\tau/2}  \|h(\tau)\|   \int_{t_{0}}^{t} e^{-\kappa(t-s)} e^{-\kappa s/2}{ds}
\notag \\
& \leq C_{0} e^{-\kappa t/2} \sup_{t_{0} \leq \tau \leq t}  e^{\kappa\tau/2} \|h(\tau)\| .
\end{align*}
Adding the inequalities above yields \eqref{S*_beta}.
The proof is complete.
\end{proof}

\section{Time-global solvability of the nonlinear problem}\label{sec5}

This section provides the time-global solvability
of {{the}} initial-boundary value problem \eqref{pe0}.
We use the solution space $X({t_{0},\kappa,\beta})$
for $t_{0}\in\mathbb R$, $\kappa \geq 0$, and $\beta>0$ as
\begin{align*}
X({t_{0},\kappa,\beta})&=\{f \in  L^{\infty}(t_{0},\infty;L^{\infty}_{\beta}) \cap C([t_{0},\infty] ; L^{2})\, | \, \tri f\tri_{t_{0},\kappa,\beta} <\infty\}, 
\\
\tri f \tri_{t_0,\kappa,\beta}&=\sup_{t_0 \leq \tau <\infty} e^{\kappa \tau } [[f(\tau)]]_{\beta}
\end{align*}
and define the mapping 
$\Phi : X({0,0,\beta}) \to X({0,0,\beta})$ as
\begin{equation*}
\Phi[\bar{g}]:=S(t)g_{0}+S*\left\{  e^{-\sigma x_{1}} \Gamma(\bar{g}) +  2\Gamma(\varphi_{R}\tilde{f}+U,\bar{g}) + e^{\sigma x_{1}}H \right\}, \quad g_{0} \in L^{2} \cap L^{\infty}_{\beta}.
\end{equation*}
We seek the fixed point $g=\Phi[g]$ which is a mild solution to \eqref{pe0}.
Note that the mild solution $g$ satisfies \eqref{pe0} in the sense of distributions.
The existence is stated in the following theorem. 
\begin{theorem}\label{global1}
Under the same assumptions as in Theorem \ref{th5},
there exists a constant $\eta>0$ such that if $[[g_{0}]]_{\beta}+\delta \leq \eta$, 
the problem \eqref{pe0} has a unique mild solution $g \in X({0,0,\beta})$ satisfying 
$\tri g \tri_{0,0,\beta} \leq C([[g_{0}]]_{\beta}+\delta )$.
\end{theorem}

To prove this, let us first show the following lemma.

\begin{lemma} \label{lemMap}
Under the same assumptions as in Theorem \ref{th5}, there hold that 
\begin{gather} 
\tri\Phi[\bar{g}]\tri_{0,0,\beta}
\lesssim [[g_{0}]]_{\beta}+\delta+(\delta+\tri\bar{g}\tri_{0,0,\beta})\tri\bar{g}\tri_{0,0,\beta},
\label{map1} \\
\tri\Phi[\bar{g}]-\Phi[\bar{h}]\tri_{0,0,\beta}
\lesssim (\delta+\tri\bar{g}\tri_{0,0,\beta}+\tri\bar{h}\tri_{0,0,\beta})\tri\bar{g}-\bar{h}\tri_{0,0,\beta}
\label{map2}
\end{gather}
for any $\bar{g}, \bar{h} \in X({0,0,\beta})$.
\end{lemma}
\begin{proof}
Let us show first \eqref{map1}. Use Lemmas \ref{lem2}--\ref{lemS*} to obtain
\begin{align*}
[[\Phi[\bar{g}](t)]]_{\beta} 
&\lesssim [[S(t)g_{0}]]_{\beta} + [[S*\left\{  e^{-\sigma x_{1}} \Gamma(\bar{g}) +  2\Gamma(\varphi_{R}\tilde{f}+U,\bar{g}) + e^{\sigma x_{1}}H \right\} ]]_{\beta} 
\\
&\lesssim [[g_{0}]]_{\beta} + \sup_{0 \leq \tau \leq t} [[ (e^{-\sigma x_{1}} \Gamma(\bar{g}) +  2\Gamma(\varphi_{R}\tilde{f}+U,\bar{g}) + e^{\sigma x_{1}}H) ]]_{\beta-1}.
\end{align*}
Then using \eqref{U1}, \eqref{H1}, Propositions \ref{ex-st} and \ref{LGamma}, we have 
\begin{equation*}
[[\Phi[\bar{g}](t)]]_{\beta} \lesssim [[g_{0}]]_{\beta}+\delta+(\delta+\tri\bar{g}\tri_{0,0,\beta})\tri\bar{g}\tri_{0,0,\beta},
\end{equation*}
which means \eqref{map1}. 
One can also show \eqref{map2} in much the same way as in the derivation of \eqref{map1}.
\end{proof}

We are now in a position to {{prove}} Theorem \ref{global1}.
\begin{proof}[Proof of Theorem \ref{global1}]
The inequality \eqref{map1} implies that there exist positive constants $\eta$ and $r$
such that if $[[g_{0}]]_{\beta}+\delta \leq \eta$, then $\Phi$ is a map from $B(r)$ to $B(r)$, where
\begin{equation*}
B(r):=\{ f \in X({0,0,\beta}) \, | \, \tri f \tri_{0,0,\beta} \leq r([[g_{0}]]_{\beta}+\delta)\}.
\end{equation*}
Furthermore, {{by taking $\eta$ small again if necessary, we see from {{the}} inequality \eqref{map2} that 
$\Phi$ is a contraction map over the Banach space $B(r)$}}. 
Therefore we have a unique fixed point $g=\Phi[g]$ in $B(r)$, {{which}} is a desired mild solution to the problem \eqref{pe0}.
\end{proof}

\section{Time-periodic and stationary solutions}\label{sec6}

In this section, we prove Theorems \ref{th4} and \ref{th5} which state
the unique existence and asymptotic stability of 
time-periodic and stationary solutions of equation \eqref{pe1}.
Let us mention the definitions of those solutions in the sense of distribution. 

\begin{definition}
We say that $g^{*}$ is a time-periodic solution with a period $T^{*}>0$ 
if it satisfies the following:
\begin{enumerate}[(i)]
\item $g^{*} \in X({-\infty,0,\beta})$.
\item $g^{*}(t+T^{*},x,\xi)=g^{*}(t,x,\xi)$ for $(t,x,\xi) \in \mathbb R \times {\mathbb R^{3}_{+}}\times \mathbb R^{3}$.
\item For any $\phi \in  \{ f \in C_{0}^{1}(\mathbb R \times \overline{\mathbb R^{3}_{+}}\times \mathbb R^{3}) \ | \  \text{$f(t,0,x',\xi)=0$ for $\xi_{1}<0$}\}$, it holds that
\begin{gather}\label{weak1}
\begin{aligned}
& (g^{*},\partial_{t}\phi)_{L^{2}_{t,x,\xi}} + ( g^{*}, \xi \cdot \nabla_{x} \phi)_{L^{2}_{t,x,\xi}} + (\sigma \xi_{1} g^{*} + Lg^{*}, \phi)_{L^{2}_{t,x,\xi}} 
\\
&\quad = - (e^{-\sigma x_{1}} \Gamma(g^{*}) +  2\Gamma(\varphi_{R}\tilde{f}+U,g^{*}) + e^{\sigma x_{1}}H,\phi)_{L^{2}_{t,x,\xi}}.
\end{aligned}
\end{gather}
\end{enumerate}
\end{definition}

\begin{definition}\label{DefS}
We say that $g^{s}$ is a stationary solution if it satisfies the following:
\begin{enumerate}[(i)]
\item $g^{s} \in X({-\infty,0,\beta})$.
\item For any $\phi \in  \{ f \in C_{0}^{1}(\overline{\mathbb R^{3}_{+}}\times \mathbb R^{3}) \ | \  \text{$f(0,x',\xi)=0$ for $\xi_{1}<0$}\}$, it holds that
\begin{gather}\label{weak2}
\begin{aligned}
& ( g^{s}, \xi \cdot \nabla_{x} \phi)_{L^{2}} + (\sigma \xi_{1} g^{s} + Lg^{s}, \phi)_{L^{2}} 
\\
&= - (e^{-\sigma x_{1}} \Gamma(g^{s}) +  2\Gamma(\varphi_{R}\tilde{f}+U,g^{s})  + e^{\sigma x_{1}}H,\phi)_{L^{2}}.
\end{aligned}
\end{gather}
\end{enumerate}
\end{definition}

\subsection{Uniqueness of time-periodic solutions}\label{sec6.1}
We first study the uniqueness of time-periodic solutions.
\begin{proposition}\label{tps1}
Suppose that the same assumptions as in Theorem \ref{th5} hold.
There exists a constant $\eta_{0}>0$ independent of $T^{*}$ 
such that if a time-periodic solution $g^{*}$ with the period $T^*>0$ satisfies 
the following inequality, then it is unique:
\begin{equation}\label{unies1}
\delta+\tri g^{*} \tri_{-\infty,0,\beta} \leq \eta_{0}.
\end{equation}
\end{proposition}
\begin{proof}
Let $g^{*}$ and $g^{\#}$ be time-periodic solutions.
Then $\overline{h} :=g^{*}-g^{\#}$ satisfies
\begin{gather*}
\begin{aligned}
& (\overline{h} ,\partial_{t} \phi)_{L^{2}_{t,x,\xi}} + ( \overline{h} , \xi \cdot \nabla_{x} \phi)_{L^{2}_{t,x,\xi}} + (\sigma \xi_{1} \overline{h}  + L\overline{h} , \phi)_{L^{2}_{t,x,\xi}} 
\\
&\quad = - (e^{-\sigma x_{1}} \{\Gamma(g^{*}) - \Gamma(g^{\#})\} +  2\Gamma(\varphi_{R}\tilde{f}+U,\overline{h}),\phi)_{L^{2}_{t,x,\xi}}.
\end{aligned}
\end{gather*}
By the energy method,  we have
\begin{align}\label{unies2}
I_{1} &\leq  \int_{0}^{T^{*}}\!\!\int_{\mathbb{R}^3_{+}} \!\! \int_{\mathbb{R}^3}  \left\{e^{-\sigma x_{1}} (\Gamma(g^{*})-\Gamma(g^{\#})) + 2\Gamma(\varphi_{R}\tilde{f}+U,\overline{h} ) \right\} \overline{h}  dxd\xi dt,
\\
I_{1}&:=- \sigma  \int_{0}^{T^{*}}\!\!\int_{\mathbb{R}^3_{+}} \!\! \int_{\mathbb{R}^3}  \xi_1 \overline{h} ^2  dxd\xi dt
- \int_{0}^{T^{*}}\!\!\int_{\mathbb{R}^3_{+}} \!\! \int_{\mathbb{R}^3}  \overline{h}  L \overline{h}  dxd\xi dt,
\notag
\end{align}
where we have used the periodicity of $\overline{h}$
and also neglected the good contribution on the boundary term.

Let us show the lower bound of $I_{1}$ as
\begin{equation}\label{unies10}
\|  \langle \xi\rangle^{1/2} \overline{h} \|_{L^{2}_{t,x, \xi}}^{2}   \lesssim I_{1}.
\end{equation}
Decomposing $L= -\nu(\xi) +K$,  using Proposition \ref{LGamma}, and taking $\sigma \leq \nu_0 /2$,  we bound $I_{1}$ as
\begin{align} \label{unies5}
I_{1} \geq  \frac{\nu_0}{2} \|\langle \xi \rangle^{1/2}\overline{h} \|_{L^2_{t, x, \xi}}^{2}  - C \|\overline{h} \|_{L^2_{t, x, \xi}}^{2}. 
\end{align}
We complete the derivation of \eqref{unies10} by showing 
that $I_{1}$ is also bound from below by
$c \|\overline{h} \|_{L^2_{t, x, \xi}}^{2}$.
For the first term of $I_{1}$, we observe that 
\begin{align}
& - \sigma \int_{0}^{T^{*}}\!\! \int_{\mathbb{R}^3_{+}} \!\! \int_{\mathbb{R}^3} \xi_1 \overline{h}^2 dxd\xi dt
\notag \\
&=-\sigma (P \xi_1 PP\overline{h}, P\overline{h})_{L^2_{t,x,\xi}}
-2\sigma (P\xi_1 (I-P)\overline{h}, P\overline{h})_{L^2_{t,x,\xi}}
-\sigma (\xi_1 (I-P)\overline{h}, (I-P)\overline{h})_{L^2_{t,x,\xi}} 
\notag \\ 
&\ge \sigma(-AP\overline{h}, P\overline{h})_{L^2_{t,x,\xi}}
-2\sigma\Vert (I-P) \overline{h}\Vert_{L^2_{t,x,\xi}}\Vert P \overline{h}\Vert_{L^2_{t,x,\xi}}
-\sigma  \Vert \xi_1^{1/2} (I-P) \overline{h} \Vert^2_{L^2_{t,x,\xi}}
\notag \\
&\ge  \nu_{2} \sigma \Vert P \overline{h} \Vert^2_{L^2_{t,x,\xi}} 
-\sigma \epsilon\Vert P \overline{h}\Vert^2_{L^2_{t,x,\xi}}
-\dfrac{\sigma}{\epsilon}\Vert  (I-P) \overline{h}\Vert^2_{L^2_{t,x,\xi}}
-\sigma  \Vert \xi_1^{1/2} (I-P)\overline{h}\Vert^2_{L^2_{t,x,\xi}},
\label{unies3}
\end{align}
where we have used Proposition \ref{LGamma} (vi) and (v) in deriving the first and second inequality, respectively,
and $\epsilon$ is a positive constant. 
For the second term of $I_{1}$, using Proposition \ref{LGamma} (ii), we obtain
\begin{align*} 
- \int_{0}^{T^{*}}\!\!\int_{\mathbb{R}^3_{+}} \!\! \int_{\mathbb{R}^3}  \overline{h} L \overline{h} dxd\xi dt 
& \geq \nu_{1}  \| \langle \xi \rangle^{1/2} (I - P) \overline{h} \|_{L^{2}_{t, x, \xi}}^{2}. 
\end{align*}
Combining this and \eqref{unies3}, we bound $I_{1}$ from below as
\begin{align} \label{unies6}
I_{1} &\geq  \nu_{2} \sigma \Vert P \overline{h} \Vert^2_{L^2_{t, x,\xi}} 
-\sigma \epsilon \Vert P \overline{h}\Vert^2_{L^2_{t, x,\xi}}
-\dfrac{\sigma}{\epsilon}\Vert  (I-P) \overline{h}\Vert^2_{L^2_{t, x,\xi}}
-\sigma  \Vert \xi_1^{1/2} (I-P)\overline{h}\Vert^2_{L^2_{t, x,\xi}}
\notag \\
& \quad + \nu_{1} \|\langle \xi\rangle^{1/2} (I - P) \overline{h} \|_{L^{2}_{t, x, \xi}}^{2} 
\notag \\
& \geq  \frac{\nu_{2} \sigma}{2}  \Vert P \overline{h} \Vert^2_{L^2_{t, x,\xi}}
+\frac{\nu_{1}}{2} \|\langle \xi\rangle^{1/2} (I - P) \overline{h} \|_{L^{2}_{t, x, \xi}}^{2} 
\end{align}
where we have also let $\epsilon=\nu_{2}/2$ and $ \sigma \leq \frac{\nu_{1}}{2}(1+2/\nu_{2})^{-1}$
in deriving the second inequality. 
From \eqref{unies5} and \eqref{unies6}, we deduce \eqref{unies10}.

Using \eqref{U1}, \eqref{unies1}, Proposition \ref{ex-st},  and Proposition \ref{LGamma} (iv), 
one can estimate the right  hand side of \eqref{unies2} as
\begin{align*}
\text{(R.H.S.)} 
& \lesssim  \int_{0}^{T^{*}}\!\! \|\langle \xi \rangle^{1/2}  (\varphi_{R}\tilde{f}+U)\|_{L^{\infty}_{x}L^{2}_{\xi}}\|\langle \xi \rangle^{1/2}\overline{h}\| \|\langle \xi \rangle^{1/2}\overline{h} \| dt 
\notag \\
&\quad +  \int_{0}^{T^{*}}\!\! (\|\langle \xi \rangle^{1/2} g^{*}\|_{L^{\infty}_{x}L^{2}_{\xi}} + \|\langle \xi \rangle^{1/2} g^{\#}\|_{L^{\infty}_{x}L^{2}_{\xi}})
\|\langle \xi \rangle^{1/2} \overline{h}\| \|\langle \xi \rangle^{1/2}\overline{h} \|  dt 
\notag \\
& \lesssim \int_{0}^{T^{*}}\!\! \|\langle \xi \rangle^{1/2}  (\varphi_{R}\tilde{f}+U)\|_{\beta}\|\langle \xi \rangle^{1/2}\overline{h}\| \|\langle \xi \rangle^{1/2}\overline{h} \|  dt 
\notag \\
&\quad +  \int_{0}^{T^{*}}\!\! (\|\langle \xi \rangle^{1/2} g^{*}\|_{\beta} + \|\langle \xi \rangle^{1/2} g^{\#}\|_{\beta})
\|\langle \xi \rangle^{1/2} \overline{h}\| \|\langle \xi \rangle^{1/2}\overline{h} \|  dt 
\notag \\
& \lesssim  \eta_0 \|\langle \xi \rangle^{1/2} \overline{h} \|_{L^{2}_{t, x, \xi}}^{2},
\end{align*}
where we have also used the fact $\|\cdot\|_{L^{\infty}_{x}L^{2}_{\xi}} \lesssim \|\cdot\|_{\beta}$.
{Substituting} this and \eqref{unies10} into \eqref{unies2} yields
$\|\langle \xi \rangle^{1/2}\overline{h}\|_{L^{2}_{t,x,\xi}}
\lesssim \eta_{0} \|\langle \xi \rangle^{1/2}\overline{h}\|_{L^{2}_{t,x,\xi}}$.
Letting $\eta_{0}$ be small, we conclude that
$\|\langle \xi \rangle^{1/2}(g^{*}-g^{\#})\|_{L^{2}_{t,x,\xi}}=0$, which means $g^{*}=g^{\#}$.
\end{proof}

\subsection{Existence of time-periodic solutions}\label{sec6.2}
For the construction of time-periodic solutions, 
we make use of the mild solution $g$ in Theorem \ref{global1}. 
Now we define
\begin{equation}\label{def1}
 g_{k}(t,x,\xi):=g(t+kT^*,x,\xi)
\quad \text{for $k=0,1,2,\ldots$.} 
\end{equation}
\begin{lemma}\label{apes}
There {exist} constants $\eta_{0}>0$  and $C_0>0$  independent of $k$ and $T^{*}$ 
such that if $[[g_{0}]]_{\beta}+\delta \leq \eta_{0}$, 
\begin{equation}\label{exiapes0}
\tri g_{l}-g_{k} \tri_{-kT^{*},\kappa/2,\beta}  \leq C_0 e^{-\kappa kT^{*}/2}
\quad \text{for $l >k$},
\end{equation}
where $\kappa$ is a positive constant being in Lemma \ref{lem2}.
\end{lemma}
\begin{proof}
We first observe by using $(U, H)(t+T^{*},x,\xi)=(U, H)(t,x,\xi)$ that 
\begin{align}
g_{l}-g_{k}
&=S(t+lT^{*})g_{0}-S(t+kT^{*})g_{0}
\notag \\
&\quad +\int_{0}^{t+lT^{*}} S(t+lT^{*}-s)\left\{  e^{-\sigma x_{1}} \Gamma(g) +  2\Gamma(\varphi_{R}\tilde{f}+U,g) + e^{\sigma x_{1}}H \right\}(s) ds
\notag \\
&=S(t+lT^{*})g_{0}-S(t+kT^{*})g_{0}
\notag \\
&\quad +\int_{-lT^{*}}^{-kT^{*}} S(t-s)\left\{  e^{-\sigma x_{1}} \Gamma(g_{l}) +  2\Gamma(\varphi_{R}\tilde{f}+U,g_{l}) + e^{\sigma x_{1}}H \right\}(s) ds
\notag \\
&\quad +\int_{-kT^{*}}^{t} S(t-s)\left\{  e^{-\sigma x_{1}} (\Gamma(g_{l}) - \Gamma(g_{k})) +  2\Gamma(\varphi_{R}\tilde{f}+U,g_{l}-g_{k})\right\}(s)ds.
\label{diff1}
\end{align}
The first three terms of the right hand side are estimated 
by using \eqref{U1}, \eqref{H1}, Propositions \ref{ex-st} and \ref{LGamma},  
Lemmas \ref{lem2} and \ref{lemS}, and Theorem \ref{global1} as 
\begin{gather}
[[S(t+lT^{*})g_{0}]]_{\beta}+[[S(t+kT^{*})g_{0}]]_{\beta} \leq C_{0} e^{-\kappa(t+kT^{*})},
\label{exiapes1} \\
\left[\left[
\int_{-lT^{*}}^{-kT^{*}} S(t-s)\left\{  e^{-\sigma x_{1}} \Gamma(g_{l}) +  2\Gamma(\varphi_{R}\tilde{f}+U,g_{l}) + e^{\sigma x_{1}}H \right\}(s) ds \right]\right]_{\beta} \leq C_{0} e^{-\kappa(t+kT^{*})},
\label{exiapes2}
\end{gather}
where $C_{0}$ is a positive constant independent of $k$ and $T^{*}$.
Let us also estimate the rightmost by using Lemma \ref{lemS*} as 
\begin{align}
&\left[\left[\int_{-kT^{*}}^{t} S(t-s)\left\{  e^{-\sigma x_{1}} (\Gamma(g_{l}) - \Gamma(g_{k})) +  2\Gamma(\varphi_{R}\tilde{f}+U,g_{l}-g_{k})\right\}(s)ds\right]\right]_{\beta} 
\notag \\
&= [[S*_{-kT^{*}}\left\{  e^{-\sigma x_{1}} (\Gamma(g_{l},g_{l}-g_{k}) + \Gamma(g_{l}-g_{k},g_{k})) +  2\Gamma(\varphi_{R}\tilde{f}+U,g_{l}-g_{k})\right\}  ]]_{\beta} 
\notag \\
&\lesssim e^{-\kappa t/2} \sup_{-kT^{*} \leq \tau \leq t} e^{\kappa \tau/2}[[(e^{-\sigma x_{1}} (\Gamma(g_{l},g_{l}-g_{k}) + \Gamma(g_{l}-g_{k},g_{k})) +  2\Gamma(\varphi_{R}\tilde{f}+U,g_{l}-g_{k}))(\tau) ]]_{\beta-1} 
\notag \\
&\lesssim e^{-\kappa t/2} (\delta+\tri g_{l}\tri_{-kT^*,0,\beta}+\tri g_{k}\tri_{-kT^*,0,\beta})\tri g_{l}-g_{k}\tri_{-kT^*,\kappa/2,\beta}
\notag \\
&\lesssim \eta e^{-\kappa t/2}  \tri g_{l}-g_{k}\tri_{-kT^*,\kappa/2,\beta} .
\label{exiapes4}
\end{align}
Take the norm $[[\ \cdot \ ]]_{\beta}$ of \eqref{diff1} and multiply the result by $e^{\kappa t/2}$.
Then using \eqref{exiapes1}--\eqref{exiapes4} and taking $\eta$ small enough, 
we arrive at \eqref{exiapes0}.
\end{proof}

We are now in a position to show the existence of time-periodic solutions.

\begin{proposition}\label{tps2}
Suppose that the same assumptions as in Theorem \ref{th5} hold.
There {exist} constants $\eta>0$  and $C_0>0$  independent of $T^{*}$ 
such that if $\delta \leq \eta$, we have a time-periodic solution $g^{*}$
with a period $T^{*}$.
Furthermore, it satisfies $\tri g^{*}\tri_{-\infty,0,\beta} \leq C_{0}\delta$.
\end{proposition}
\begin{proof}
We take the mild solution $g$ for the initial data $g_{0}$ satisfying $[[g_{0}]]_{\beta} \lesssim \delta$, which is established in Theorem \ref{global1}, 
and define $g_{k}$ as \eqref{def1}.
From Lemma \ref{apes}, there exists $g^{*} \in X(-\infty,0,\beta)$ such that
\begin{equation}\label{convergent1}
g_{k} \to g^{*}  \quad \text{in  $L^{\infty}_{\rm loc}(\mathbb R; L^{\infty}_{\beta}) \cap C_{\rm loc}(\mathbb R; L^{2})$}.
\end{equation}
Furthermore we have $g_{k}(t+T^{*},x,\xi)=g_{k+1}(t,x,\xi)$ from the definition of $g_{k}$ 
and hence it holds that $g^{*}(t+T^{*},x,\xi)=g^{*}(t,x,\xi)$.
This periodicity together with \eqref{convergent1} and Theorem \ref{global1} ensure that
$\tri g^{*}\tri_{-\infty,0,\beta} \leq C_{0}\delta$, where $C_{0}$ is a positive constant independent of $T^{*}$.

What is left is to check that $g^{*}$ satisfies the weak form \eqref{weak1}.
Taking $k$ large according to a fixed $\phi \in C_{0}^{1}(\mathbb R \times \overline{\mathbb R^{3}_{+}}\times \mathbb R^{3})$ with $\phi(t,0,x',\xi)=0$ if $\xi_{1}<0$,  we see that $g_{k}$ satisfies 
\begin{align*}
& (g_{k},\partial_{t}\phi)_{L^{2}_{t,x,\xi}} + ( g_{k}, \xi \cdot \nabla_{x} \phi)_{L^{2}_{t,x,\xi}} + (\sigma \xi_{1} g_{k} + Lg_{k}, \phi)_{L^{2}_{t,x,\xi}} 
\notag \\
&\ \ = - (e^{-\sigma x_{1}} \Gamma(g_{k}) +  2\Gamma(\varphi_{R}\tilde{f}+U,g_{k}) + e^{\sigma x_{1}}H,\phi)_{L^{2}_{t,x,\xi}}.
\end{align*}
Now letting $k \to \infty$
and using \eqref{U1}, \eqref{H1}, \eqref{convergent1}, and  Propositions \ref{ex-st} and \ref{LGamma}, we obtain
\begin{align*}
& (g^{*},\partial_{t}\phi)_{L^{2}_{t,x,\xi}} + ( g^{*}, \xi \cdot \nabla_{x} \phi)_{L^{2}_{t,x,\xi}} + (\sigma \xi_{1} g^{*} + Lg^{*}, \phi)_{L^{2}_{t,x,\xi}} 
\notag \\
&\ \ = - (e^{-\sigma x_{1}} \Gamma(g^{*}) +  2\Gamma(\varphi_{R}\tilde{f}+U,g^{*}) + e^{\sigma x_{1}}H,\phi)_{L^{2}_{t,x,\xi}}.
\end{align*}
The proof is complete.
\end{proof}

\subsection{Stationary solutions}\label{sec6.3}
We show that the time-periodic solutions in Proposition \ref{tps2} are time-independent
if $U$ and $H$ are time-independent.
\begin{proposition}\label{ss1}
If $U$ and $H$ are time-independent, 
the time-periodic solution in Proposition \ref{tps2} is a stationary solution.
\end{proposition}
\begin{proof}
Proposition \ref{tps2} ensures the existence of
time-periodic solutions $g^{*}$ for any period $T^*$, since $U$ and $H$ are time-independent.
We remark that the smallness assumption for $\delta$ is independent of the period $T^*$.
Hence, one can have time-periodic solutions $g^{*}$ with the period $T^*$ and
$g_{l}^{*}$ with the period $T^{*}/2^l$ for $l \in \mathbb N$ 
under the same assumption on $\delta$.
Furthermore, $g^{*}=g_{l}^{*}$ follows from 
Proposition \ref{tps1}, since $g^{*}$ and $g_{l}^{*}$ are 
the time-periodic solutions with the period $T^{*}$ and satisfy \eqref{unies1}.
Now we know that
\[
g^{*}\left(0,x,\xi\right)=
g^{*}\left(\frac{i}{2^l}T^{*},x,\xi\right)
\quad \text{for $i=1,2,3,\ldots,2^l$ and $l=0,1,2,\ldots$.} 
\]
Because the set 
$\cup_{l \geq 0} \{{i}/{2^l} \ ; \ i=1,2,3,\ldots,2^l\}$
is dense in $[0,T^{*}]$,
we see from $g^{*} \in C(\mathbb R; L^{2})$
that $g^{*}$ is independent of $t$.
Now it is straightforward to check that $g^{*}$ satisfies the conditions in Definition \ref{DefS}.
Therefore, $g^{s}=g^{*}$ is the desired stationary solution.
\end{proof}

Now Theorem \ref{th4} immediately follows from 
Propositions \ref{tps1}, \ref{tps2}, and \ref{ss1}.

\subsection{Asymptotic stability}\label{sec7}
Finally, we prove the asymptotic stability of {{the}} time-periodic and stationary solutions. 
Any mild solution $g$ in Theorem \ref{global1} satisfies \eqref{exiapes0} with $k=0$. 
It means that
\begin{equation*}
[[(g_{l}-g)(t)]]_{\beta} \lesssim e^{-\kappa t/2} \quad \text{for $t>0$}.
\end{equation*}
Letting $l \to \infty$ and using \eqref{convergent1},  we have 
\begin{equation*}
[[(g^{*}-g)(t)]]_{\beta} \lesssim e^{-\kappa t/2} \quad \text{for $t>0$}.
\end{equation*}
Therefore we see that the time-periodic and stationary solutions are asymptotically stable. 

This fact and Theorem \ref{global1} together imply Theorem \ref{th5}. 
Note that Theorem \ref{global1} ensures the time-global solvability of the initial-boundary value problem \eqref{pe0}.

\subsection*{Acknowledgements}
{The first author is supported by JSPS KAKENHI Grant (No.\ 20K14338). The second author is supported by JSPS KAKENHI Number 18K03364. The third author is supported by AMS-Simons Foundation for an AMS-Simons Travel Grant.

\providecommand{\bysame}{\leavevmode\hbox to3em{\hrulefill}\thinspace}
\providecommand{\MR}{\relax\ifhmode\unskip\space\fi MR }

\end{document}